\documentclass[a4paper,11pt,reqno]{amsart}

\usepackage[latin1]{inputenc}
\usepackage{amsmath,amssymb,amsthm,amsfonts,mathrsfs,amsopn} 
\usepackage{dsfont}   
\usepackage{mathtools}
\mathtoolsset{showonlyrefs}  

\usepackage{
amssymb,
tikz,
xcolor,
}

\usepackage[hidelinks]{hyperref} 
\usepackage{bbm}
\usepackage{tikz-cd}    
\usepackage{geometry}
\usepackage{slashed}

\newtheorem{thm}{Theorem}[section]
\newtheorem{lemma}[thm]{Lemma}
\newtheorem{proposition}[thm]{Proposition}
\newtheorem{corollary}[thm]{Corollary}

\theoremstyle{remark}
\newtheorem{remark}[thm]{Remark}

\newtheorem{defi}[thm]{Definition}

\numberwithin{equation}{section}

\newcommand{\p}{\partial}
\newcommand{\R}{\mathbb{R}}
\newcommand{\C}{\mathbb{C}}
\newcommand{\D}{\slashed{D}} 
\newcommand{\Penr}{\slashed{P}} 
\newcommand{\dd}{\mathop{}\!\mathrm{d}}
\newcommand{\dv}{\dd{\vol}}
\newcommand{\sph}{\mathbb{S}} 
\newcommand{\zero}{\mathcal{Z}} 
\newcommand{\I}{\mathrm{i}} 
\newcommand{\eps}{\varepsilon}  

\newcommand{\Abracket}[1]{\left<#1\right>} 
\newcommand{\parenthesis}[1]{\left(#1\right)} 
\newcommand{\braces}[1]{\left\{#1\right\}} 

\DeclareMathOperator{\Vol}{Vol}
\DeclareMathOperator{\Spin}{Spin}

\DeclareMathOperator{\SO}{SO}
\DeclareMathOperator{\End}{End}
\DeclareMathOperator{\Id}{Id}
\DeclareMathOperator{\Scal}{Scal} 
\DeclareMathOperator{\grad}{grad} 
\DeclareMathOperator{\capacity}{cap}
\DeclareMathOperator{\dist}{dist}
\DeclareMathOperator{\Span}{Span}
\DeclareMathOperator{\Conf}{Conf} 
\DeclareMathOperator{\Killing}{\mathscr{K}} 
\DeclareMathOperator{\YM}{Y} 

\renewcommand{\geq}{\geqslant}
\renewcommand{\leq}{\leqslant}

\newcommand\cH{{\mathcal{H}}}

\newcommand\cB{{\mathcal{B}}}
\newcommand\cA{{\mathcal{A}}}

\newcommand\cL{{\mathcal{L}}}

\newcommand\vol{{\operatorname{vol}}}

\newcommand{\be}{\begin{equation}}%
\newcommand{\ee}{\end{equation}}

\newcommand{\blu}[1]{\textcolor{blue}{#1}}

\title[Ground state Dirac bubbles and Killing spinors]{Ground state Dirac bubbles and Killing spinors}

\author[W. Borrelli]{William Borrelli}
\address[W. Borrelli]{Scuola Normale Superiore, Centro De Giorgi, Piazza dei Cavalieri 3, I-56100 , Pisa, Italy.} 
\email{william.borrelli@sns.it}

\author[A. Malchiodi]{Andrea Malchiodi}
\address[A. Malchiodi]{Scuola Normale Superiore, Piazza dei Cavalieri 7, I-56100 , Pisa, Italy.} 
\email{andrea.malchiodi@sns.it}

\author[R. Wu]{Ruijun Wu}
\address[R. Wu]{International School for Advanced Studies (SISSA), Via Bonomea 265, I-34136, Trieste, Italy}
\email{rwu@sissa.it}

\date{\today}

\begin{document}

\begin{abstract}
 We prove a classification result for ground state solutions of the critical Dirac equation on $\R^n$, $n\geq2$. 
 By exploiting its conformal covariance, the equation can be posed on the round sphere~$\sph^n$ and the non-zero solutions at the ground level are given by Killing spinors, up to conformal diffeomorphisms.  
 Moreover, such ground state solutions of the critical Dirac equation are also related to the Yamabe equation 
 for the sphere, for which we crucially exploit some known classification results. 
\end{abstract}

\maketitle

{\footnotesize
\emph{Keywords}: critical Dirac equations, ground state solutions, Dirac bubbles, Killing spinors, Yamabe problem

\medskip

\emph{2010 MSC}: 53C27, 58J90, 81Q05.
}


\section{Introduction}
\label{sec:intro}
\subsection{Main results.}
  
We are interested in the following nonlinear  Dirac equation
\be\label{eq:nld}
\D\psi=\vert\psi\vert^{2^\sharp-2}\psi \qquad\mbox{on}\quad\R^{n}\,,\, n\geq2\,,
\ee
with \emph{critical exponent}
$$
2^\sharp:=\frac{2n}{n-1}. 
$$
This equation arises naturally in conformal spin geometry and in variational problems related to critical Dirac equations on spin manifolds. 
Moreover, two-dimensional critical Dirac equations recently attracted a considerable attention as effective equations for wave propagation in honeycomb structures, as explained in Section \ref{sec:motivation}. 
\medskip

We consider solutions to ~\eqref{eq:nld} corresponding to critical points of the following functional  
\be\label{eq:action}
\cL(\psi)=\frac{1}{2}\int_{\R^{n}}\langle\D\psi,\psi\rangle\dd{\vol_{g_{\R^n}}}-\frac{1}{2^\sharp}\int_{\R^n}\vert\psi\vert^{2^\sharp}\dd{\vol_{g_{\R^n}}}\,,
\ee
belonging to the homogeneous Sobolev space $\mathring{H}^{1/2}(\R^n,\C^N)$, which is the completion of the space $C^\infty_c(\R^n,\C^N)$ with respect to $\Vert\psi\Vert^2_{\mathring{H}^{1/2}}:=\int_{\R^n}\vert\xi\vert\vert\widehat{\psi}(\xi)\vert^2\,\dd\xi$. Here $\widehat{\psi}$ denotes the Fourier transform of $\psi$ and $N=2^{[\frac{n}{2}]}$.
\smallskip

The following lower bound for non-zero solutions has been proved in \cite{Isobecritical}: 
\be\label{eq:lowerbound}
\cL(\psi)\geq\frac{1}{2n}\left(\frac{n}{2}\right)^n\omega_n\,, 
\ee
where~$\omega_n=\Vol_{g_0}(\sph^n)$ denotes the volume of the round unit $n$-sphere~$\sph^n\subset\R^n$.

As it will be explained in Section \ref{sec:covariance}, both the functional \eqref{eq:action} and  equation \eqref{eq:nld} are conformally invariant so that one can equivalently study it on the $n$-dimensional unit sphere $\sph^n$, 
\be\label{eq:nlds}
\D_{g_0}\psi=\vert\psi\vert^{2^{\sharp}-2}\psi\qquad\mbox{on}\;(\mathbb{S}^{n},g_0)\,,
\ee
where $\sph^n$ is endowed with the round metric $g_0$ and its canonical spin structure. As a consequence, inequality \eqref{eq:lowerbound} also holds for the functional \eqref{eq:action} on the round sphere, denoted by $\mathcal{L}_{g_0}$. 

\begin{defi}
We say that a \emph{non-trivial} solution $\psi\in H^{1/2}(\sph^n,\Sigma_{g_0}\sph^n)$ to \eqref{eq:nlds} is a \emph{ground state solution} if equality in \eqref{eq:lowerbound} holds, that is
\be\label{eq:gsdef}
\cL_{g_0}(\psi)=\frac{1}{2n}\left(\frac{n}{2} \right)^n\omega_n.
\ee
\end{defi}

Our main result is the following
\begin{thm}\label{thm:main}
Let $\psi\in H^{1/2}(\sph^n,\Sigma\sph^n)$ be a ground state solution to \eqref{eq:nlds} with $n\geq2$. 
Then,~$\psi$ is a~$(-\frac{1}{2})$-Killing spinor up to a conformal diffeomorphism. 
More precisely, there exists a ~$(-\frac{1}{2})$-Killing spinor~$\Psi\in\Gamma(\Sigma_{g_0}\sph^n)$ and a conformal diffeomorphism~$f\in \Conf(\sph^n,g_0)$ such that
\begin{equation} 
 \psi=\parenthesis{\det(\dd f)}^{\frac{n-1}{2n}}\beta_{f^*g_0,g_0}(f^*\Psi), 
\end{equation}
where~$\beta_{f^*g_0,g_0}$ is the identification of spinor bundles for conformally related metrics. 
\end{thm} 

We will give more details on the pullback $f^*$ in Section~\ref{sect:conformal diffeomorphism}. 

\begin{remark}\label{rmk:ammannclassification}
In \cite{ammannsmallest}, B. Ammann studied (actually, on a general spin manifold $M$) the conformally invariant functional
\be\label{eq:ammannaction}
\mathcal{F}^g_{q_D}(\varphi):=\frac{\int_{\sph^n}\langle\D_g\varphi,\varphi\rangle\dd{\vol_{g_0}} }{\Vert\D_g\varphi\Vert^2_{L^{q_{_{_D}}}}}\,,
\ee
where $q_{_D}=\frac{2n}{n+1}$ is the conjugate exponent of $2^\sharp$.  

He showed that \eqref{eq:ammannaction} is well-defined and bounded above on $W^{1,q_{_D}}(\sph^n,\Sigma\sph^n)$: assuming some extra regularity, he proved that any maximizer $\phi$ is of the form $\varphi=f^*\Psi$, where $\Psi$ is a (-1/2)-Killing spinor and $f:\sph^n\to\sph^n$ is an orientation preserving conformal diffeomorphism.

The Sobolev-like quotient \eqref{eq:ammannaction} is closely related to the functional \eqref{eq:action}. Indeed, suitably choosing the functional setting, it should be possible to prove that those functionals are related by a duality relation, but we prefer not to investigate this aspect here. We observe that our main result Theorem \ref{thm:main} deals with critical points of \eqref{eq:action} under minimal regularity assumptions, proving an analogous classification result. To this aim we need a careful analysis of the nodal set, as stated in Theorem \ref{thm:nodalset}.
\end{remark}

The ground state solutions of~\eqref{eq:nld} on~$\R^n$ are obtained via pulling back the above spinors via  stereographic projection.
\begin{corollary}\label{cor:bubbles on Rn}
 Let~$\psi_{_{\R^n}}\in\mathring{H}^{\frac{1}{2}}(\R^n,\C^N)$ be a ground state solution of~\eqref{eq:nld}. 
 Then there exists $\widetilde{\Phi}_0\in\C^N$ with~$|\widetilde{\Phi}_0|=\frac{1}{\sqrt{2}}\parenthesis{\frac{n}{2}}^{\frac{n-1}{2}}$, and~$x_0\in\R^n$, ~$\lambda>0$ such that 
 \begin{align}
 \psi_{_{\R^n}}(x)
 =\parenthesis{\frac{2\lambda}{\lambda^2+|x-x_0|^2}}^{\frac{n}{2}} \parenthesis{\mathds{1}-\gamma_{_{\R^n}}\parenthesis{\frac{x-x_0}{\lambda}}}\widetilde{\Phi}_0\,.  
\end{align} 
\end{corollary}
\smallskip

On the other hand, for $n=2$ infinitely many explicit \emph{excited state solutions} (i.e. solutions for which inequality \eqref{eq:lowerbound} is strict) have been found in \cite{borrellifrank}. We also mention that the ground state solutions for the Dirac-Einstein equations have been recently classified in \cite{Borrelli-Maalaoui-JGA2020}. 
\begin{remark}
There is a similar statement in the Yamabe problem, namely, up to conformal diffeomorphisms, the prescribing scalar curvature equation
\begin{equation}\label{eq:Yamabe-n(n-1)}
 -4\frac{n-1}{n-2}\Delta_{g_0} h+\Scal_{g_0} h
 =n(n-1)h^{\frac{n+2}{n-2}} 
 \qquad \mbox{ on } \quad (\sph^n,g_0)
\end{equation}
admits a unique positive solution in~$H^1(\sph^n)$ given by the constant function $h\equiv1$. 
Geometrically this can be reformulated as Obata's Theorem \cite[Theorem 6.6]{obata1971theconjectures}: the round metric $g_0$ is the only (up to conformal diffeomorphisms) metric on $\sph^n$ which has constant scalar curvature~$n(n-1)$. 
Indeed, this fact will be used in the proof of our result. 
\end{remark}

The proof of Theorem \ref{thm:main} in the case $n\geq 3$ requires an estimate on the Hausdorff dimension of the nodal set of solutions. 
\begin{thm}\label{thm:nodalset}
Let $n\geq 3$ and $\psi\in H^{1/2}(\sph^n,\Sigma\sph^n)$ be a non-trivial solution to \eqref{eq:nlds}.
The nodal set 
\be\label{eq:nodalset}
\zero(\psi):=\{x\in\sph^n\,:\,\psi(x)=0 \}
\ee
has Hausdorff dimension at most $n-2$.
\end{thm}
The above theorem generalizes B\"ar's result \cite{bar0}, which holds for equations of the form $\D\psi=V(\psi)$ on a spin manifold $M$, where $V:\Sigma M\to\Sigma M$ is a \emph{smooth} fiber-preserving map of the spinor bundle. 
We remark that it is indeed the case for \eqref{eq:nld} when $n=2$, but this is not necessarily true in higher dimension, as smoothness of solutions is not guaranteed in that case. Indeed, for $n\geq 3$ the function $x\mapsto \vert x\vert^{\frac{2}{n-1}}$ is not smooth at $x=0$. Then, as a consequence, solutions are a priori only of class $C^{1,\alpha}$ near the nodal set, for all $0<\alpha<1$, (see \cite{borrellifrank,Isobecritical}). Note that away from the nodal set solutions are smooth by standard regularity theory.  

\subsection{Some motivations.}	\label{sec:motivation}

Equation \eqref{eq:nld} appears in the study of different problems from conformal geometry and mathematical physics, as shortly explained in this section. 

It describes, for instance, the blow-up profiles for the equation
\be\label{criticaldiracmanifold}
\D\psi=\mu\psi+\vert\psi\vert^{2^\sharp -2}\psi \qquad\mbox{on}\ M,\,\mbox{with}\, \mu\in\R\,,
\ee
where $(M,g)$ is a compact spin manifold. For $\mu=0$ the equation is usually referred to as the \emph{spinorial Yamabe equation} and has been studied in \cite{Ammann,ammannsmallest,spinorialanalogue,ammannmass}; see also \cite{nadineconformalinvariant,nadineboundedgeometry,maalaoui,raulot} and references therein. Equation \eqref{criticaldiracmanifold} with general $\mu\in\R$ is the spinorial analogue of the Br\'ezis--Nirenberg problem \cite{brezisnirenberg} and has been studied, for instance, in \cite{bartschspinorial} and \cite{Isobecritical}.

Note that solutions of \eqref{criticaldiracmanifold} are critical points of the functional
\be\label{functionalmanifold}
\cL(\psi)=\frac{1}{2}\int_{M}\langle\D\psi,\psi\rangle d\operatorname{vol_g}-\frac{\mu}{2}\int_{M}\vert\psi\vert^2d\operatorname{vol_g}-\frac{1}{2^\sharp}\int_{M}\vert\psi\vert^{2^\sharp}d\operatorname{vol_g}
\ee
defined on $H^{\frac{1}{2}}(\Sigma M)$, the space of $H^{\frac{1}{2}}$-sections of the spinor bundle $\Sigma M$ of the manifold, see Sect~\ref{subsect:fractional Sobolev}. Then by \cite[Theorem 5.2]{Isobecritical} any Palais--Smale sequence $(\psi_n)_{n\in\mathbb{N}}\subseteq H^{\frac{1}{2}}(\Sigma M)$ for the functional $\cL$, up to subsequences, satisfies
\be\label{psdecomposition}
\psi_{n}=\psi_\infty+\sum^{N}_{j=1}\omega^{j}_{n}+o(1) \qquad\mbox{in $H^{\frac{1}{2}}(\Sigma M)$},
\ee
where $\psi_{\infty}$ is the weak limit of $(\psi_{n})_n$ and the $\omega^{j}_{n}$ are suitably rescaled versions of solutions to \eqref{eq:nld}. 
This is the spinorial counterpart of Struwe's theorem for the Br\'ezis-Nirenberg problem \cite{struwedecomposition}. 
Thus, equation \eqref{eq:nld} describes bubbles in the spinorial Yamabe and Br\'ezis--Nirenberg problems.
\medskip

Critical Dirac equations also appear as effective models for the wave propagation in two-dimensional honeycomb structures. Assume that $V\in C^{\infty}(\R^{2},\R)$ possesses the symmetries of a honeycomb lattice. As proved in \cite{FWhoneycomb}, the dispersion bands of the associated Schr\"{o}dinger operator
$$
H=-\Delta+V(x)
\qquad\text{in}\ L^2(\R^2)
$$ 
exhibit generically conical intersections (called \textit{Dirac points}). Then the Dirac operator turns out to be the effective operator describing the dynamics of wave packets spectrally concentrated around those conical points. Consider a wave packet $u_{0}(x)=u^{\varepsilon}_{0}(x)$ spectrally concentrated around a Dirac point, that is,
\be\label{concentrated}
u^{\varepsilon}_{0}(x)=\sqrt{\varepsilon}(\psi_{0,1}(\varepsilon x)\Phi_{1}(x)+\psi_{0,2}(\varepsilon x)\Phi_{2}(x))
\ee
where $\Phi_{j}$, $j=1,2$, are the Bloch functions at a Dirac point and the functions $\psi_{0,j}$ are some modulation amplitudes. The solution to the nonlinear Schr\"{o}dinger equation with parameter~$\kappa\in\R\setminus\{0\}$,
\be\label{gp}
\I \partial_{t}u=-\Delta u+V(x)u+\kappa\vert u\vert^{2}u \,,
\ee
and with initial conditions $u_0^\epsilon$ evolves to leading order in $\eps$ still as a modulation of Bloch functions,
\be\label{approximate}
u^{\varepsilon}(t,x)\underset{\epsilon\rightarrow0^{+}}{\sim}\sqrt{\varepsilon}\left(\psi_{1}(\varepsilon t,\varepsilon x)\Phi_{1}(x)+\psi_{2}(\varepsilon t,\varepsilon x)\Phi_{2}(x) + \mathcal{O}(\varepsilon)\right).
\ee
Fefferman and Weinstein in \cite{wavedirac} pointed out that the modulation coefficients $\psi_{j}$ satisfy the following effective Dirac system,
\begin{equation}\label{effective}
\left\{\begin{aligned}
    \partial_{t}\psi_{1}+\overline{\lambda}(\partial_{x_{1}}+\I \partial_{x_{2}})\psi_{2} &=-\I \kappa(\beta_{1}\vert\psi_{1}\vert^{2}+2\beta_{2}\vert\psi_{2}\vert^{2})\psi_{1} \,, \\
     \partial_{t}\psi_{2}+\lambda(\partial_{x_{1}}-\I \partial_{x_{2}})\psi_{1} &=-\I \kappa(\beta_{1}\vert\psi_{2}\vert^{2} + 2\beta_{2}\vert\psi_{1}\vert^{2})\psi_{2} \,,
\end{aligned}\right.
\end{equation}
where $\beta_{1,2}>0$ and $\lambda\in\C\setminus\{0\}$ are coefficients related to the potential $V$. The large-time validity of the Dirac approximation has been proved in \cite{FWwaves} in the linear case $\kappa=0$ and in \cite{arbunichsparber} for cubic nonlinearities.

For stationary solutions, i.e.  $\partial_t\psi_1=\partial_t\psi_2=0$, and for a suitable choice of the parameters involved, \eqref{effective} reduces to \eqref{eq:nld} with $n=2$. Existence and regularity of solutions to \eqref{effective} of `vortex-type' (for general values of $\beta_{1,2}, \lambda$) have been investigated in \cite{massless,borrellifrank}. In particular, in \cite{borrellifrank} existence and uniqueness of such solutions (among spinors of the same form) are proved under suitable boundary conditions at the origin. We also mention the papers \cite{shooting,Multiplicity-2020}, where the \emph{massive} case is addressed.

\subsection{Outline of the paper.}
In Section \ref{sec:preliminaries}, we first recall some preliminaries and also fix our notation. 
Exploiting some results from the literature, we give a short proof of Theorem~\ref{thm:main} for the two-dimensional case in Section \ref{sec:twodimension}. 
Then, assuming the validity of Theorem \ref{thm:nodalset}, we prove Theorem~\ref{thm:main} in dimension $n\geq3$ in Section \ref{sec:higherdimension}, with a particular emphasis on the nodal set of the solution.
Finally, Section \ref{sec:nodalset} is devoted to the proof of Theorem~\ref{thm:nodalset}, which gives an estimate for the dimension of the nodal set of solutions, thus completing the proof of Theorem \ref{thm:main} for $n\geq3$.

\smallskip

\noindent{\bf Acknowledgements.}
The authors are grateful to B. Ammann for bringing to their attention some results contained in \cite{ammannsmallest} and to G. Buttazzo for pointing out reference \cite{Swanson-Ziemer99} to them.

A.M. has been partially supported by the project {\em Geometric problems with loss of compactness}  from Scuola Normale Superiore and by MIUR Bando PRIN 2015 2015KB9WPT$_{001}$.
A.M. and W.B. are members of GNAMPA as part of INdAM and are supported by the GNAMPA 2020 project \emph{Aspetti variazionali di alcune PDE in geometria conforme}. W.B. and R.W. are also supported by Centro  di  Ricerca  Matematica  \emph{Ennio  de  Giorgi}. 
 
\section{Preliminaries}\label{sec:preliminaries}
 
In this section we collect some known facts in spin geometry and on  analytical properties of Dirac operators. 
For more details on spin geometry and the Dirac operator one can refer to ~\cite{Ammann,diracspectrum, Jost,LawsonMichelsohn}.   

\subsection{Spin structures}
Let $(M,g)$ be an oriented Riemannian manifold of dimension~$n\ge 2$.

Recall that the special orthogonal group~$\SO(n)$ has non-trivial fundamental group: ~$\pi_1(\SO(2))\cong\mathbb{Z}$ and~$\pi_1(\SO(n))=\mathbb{Z}_2$ for~$n\ge 3$.  
Thus there exist double coverings for any~$ n\ge 2$, given by the so-called spin groups: 
\begin{equation}
 \lambda:\Spin(n)\to \SO(n).
\end{equation}

\begin{defi}
A \textit{spin structure} on $(M,g)$ is a pair $(P_{\Spin} (M,g),\sigma)$, where $P_{\Spin} (M,g)$ is a $\Spin(n)$-principal bundle and $\sigma:P_{\Spin} (M,g)\rightarrow P_{\SO}(M,g)$ is a 2-fold covering map, which restricts to the non-trivial covering $\lambda:\Spin(n)\rightarrow SO(n)$ on each fiber. 
In other words, the quotient of each fiber by $\{-1,1\}\simeq\mathbb{Z}_{2}$ is isomorphic to the frame bundle of $M$, so that the following diagram commutes
\begin{center}
 \begin{tikzcd}
 P_{\Spin}(M,g)\arrow[rr,"\sigma"]\arrow[dr]& &P_{\SO}(M,g)\arrow[dl]\\
 &M & 
 \end{tikzcd}
\end{center}
A Riemannian manifold $(M,g)$ endowed with a spin structure is called a \textit{spin manifold}.
 \end{defi}
It is well-known that an orientable manifold admits a spin structure if and only if its second Stiefel--Whitney class vanishes; and in that case the spin structure needs not to be unique: the different spin structures are parametrized by elements in~$H^1(M;\mathbb{Z}_2)$. 
In particular, the spin structures of the Euclidean space $(\R^{n},g_{\R^n})$ and of the round sphere $(\mathbb{S}^{n},g_0)$, with $n\geq 2$, are actually unique.
 
\begin{defi}
The \emph{complex spinor bundle} $\Sigma M\to M$ is the vector bundle associated to the $\Spin(n)$-principal bundle $P_{\Spin}(M,g)$ via the complex spinor representation of~$\Spin(n)$.  
\end{defi}
The complex spinor bundle~$\Sigma M$ has rank $N=2^{[\frac{n}{2}]}$.  
It is endowed with a canonical spin connection~$\nabla^s$ (which is a lift of the Levi-Civita connection) and an Hermitian metric~$g^s$ which will be abbreviated as~$\left<\cdot, \cdot\right>$ if there is no confusion.  

\subsection{The Dirac operator and special spinors}
On the spinor bundle~$\Sigma M$ there is a \emph{Clifford map}~$\gamma\colon TM\to \End_\C(\Sigma M)$ which satisfies the \emph{Clifford relation}
\begin{equation}
 \gamma(X)\gamma(Y)+\gamma(Y)\gamma(X)=-2g(X,Y)\Id_{\Sigma M}, 
\end{equation}
for any tangent vector fields~$X,Y\in\Gamma(TM)$. 
The Clifford map is compatible with the Hermitian metric~$g^s$ above in the sense that
\begin{equation}
 \left<\gamma(X)\psi, \gamma(X)\varphi\right>_{g^s}
 =g(X,X)\left<\psi,\varphi\right>_{g^s},
 \qquad \forall X\in\Gamma(TM),
 \quad \forall \psi,\varphi\in\Gamma(\Sigma M). 
\end{equation}

Locally, taking an oriented orthonormal frame $(e_i)$ with dual frame $(e^i)$, the \emph{Dirac} operator $\D$ and the \emph{Penrose} operator $\Penr$, respectively, as 
\[
\D:\Gamma(\Sigma M)\to \Gamma(\Sigma M)\,,\qquad \D\psi:=\gamma(e_i)\nabla^s_{e_i}\psi\,,
\]
\[
\Penr:\Gamma(\Sigma)\to\Gamma(T^*M\otimes\Sigma M)\,,\qquad \Penr\psi\coloneqq \nabla^s\psi+\frac{1}{n}e^i\otimes\gamma(e_i)\D\psi 
\] 

Here and in the sequel, we always use the Einstein summation convention.
\begin{defi}
The spinors in~$\ker(\D)$ are called~\emph{harmonic spinors}, while those in~$\ker(\Penr)$ are called \emph{twistor spinors}.
\end{defi}
For any~$\psi\in\Gamma(\Sigma_g M)$, we have the following~\emph{pointwise Penrose--Dirac} decomposition
\begin{equation}\label{eq:Penrose-Dirac decomposition}
 |\nabla^s\psi|^2
 =|\Penr^g\psi|^2+\frac{1}{n}|\D^g\psi|^2. 
\end{equation}
The spinors which are twistor spinors and at the same time eigenspinors of~$\D$ deserve special interest: they are the so-called Killing spinors, defined as follows. 
\begin{defi}
 Given~$\alpha\in\C$, a non-zero spinor field~$\psi\in\Gamma(\Sigma M)$ is called~$\alpha$-Killing if 
 \begin{equation}\label{eq:killingequation}
  \nabla^s_X\psi=\alpha\gamma(X)\psi,\qquad \forall X\in\Gamma(TM).  
 \end{equation}
\end{defi}
The~$\alpha$-Killing spinors form a vector space, denoted by~$\Killing(g;\alpha)$.
The name comes from the fact that real Killing spinors give rise to Killing (tangent) vector fields: if~$\alpha\in\R$, then the vector field defined by
\begin{equation}
 g(V,X)\coloneqq \sqrt{-1}\left<\psi, \gamma(X)\psi\right>, \qquad \forall X\in\Gamma(TM)
\end{equation}
is a Killing field on~$(M^n,g)$.
Hence Killing spinors only exists on manifold with infinitesimal symmetries. 
For more information on Killing spinors and twistor spinors, we refer to~\cite[Appendix A]{diracspectrum}.  

Observe that on Euclidean space Killing spinors are exactly given by constant vector-valued functions $\psi\colon \R^n\to \C^N$, with~$\alpha=0$. 
The case of spheres is particularly relevant for our purposes. 

\begin{proposition}[Killing spinors on round spheres, {\cite[Appendix]{diracspectrum}}]\label{prop:killing-on-spheres}
Let $(\sph^n,g_0)$ be the standard $n(\ge 2)$-sphere in $\R^{n+1}$ and consider an~$\alpha$-Killing spinor $\psi$, for some~$\alpha\in\C$. 
 \begin{enumerate}
  \item[1.] The zero set of $\psi$ is empty. Moreover, $\alpha\in\{\pm1/2\}$,  and $\psi$ has constant length: $|\psi|\equiv const$. 
  \item[2.] The space of $1/2$-Killing spinors is $2^{[\frac{n}{2}]}$-dimensional. Such spinors are given by $\varphi=\Phi|_{\sph^n}$, where $\Phi$ is a constant spinor on $\R^{n+1}$. They coincide with eigenspinors for the first negative eigenvalue $\lambda_{-1}= -n/2$ of $\D$.
  \item[3.] The space of $-1/2$-Killing spinors is $2^{[\frac{n}{2}]}$-dimensional. Such spinors are given by $\xi=\Psi|_{\sph^n}$, where $\Psi(x)=\gamma(x)\Phi$~$(\forall x\in\R^{n+1})$ is a non-parallel twistor spinor on~$\R^{n+1}$, $\Phi$ as above. 
  They coincide with eigenspinors for the first positive eigenvalue~$\lambda_{1}= n/2$ of~$\D$.
 \end{enumerate}
\end{proposition} 
In the paper \cite{Lu-Pope-Rahmfeld-1999} Killing spinors on $(\sph^n,g_0)$ are explicitly computed in spherical coordinates.  
\subsection{Sobolev spaces of spinors}\label{subsect:fractional Sobolev} 
We recall that since $(M,g,\sigma)$ is a \emph{compact} spin manifold the spectrum of the Dirac operator is discrete and unbounded on both sides of $\R$, accumulating at $\pm\infty$. Then, using the spectral decomposition of $\D$ one can define fractional order Sobolev spaces of spinors.

Embedding theorems of Sobolev spaces into Lebesgue and H\"older spaces of spinors, analogous to the Euclidean case, also hold. We refer the reader to \cite[Section 3]{Ammann} for more details.


\subsection{Conformal symmetry}\label{sec:covariance}
Of particular importance for us is the behavior of the Dirac and Penrose operators under conformal transformations of the metric, see  e.g.~\cite{hitchin,LawsonMichelsohn, diracspectrum,jost2018symmetries}. 
To make this clear we label the various geometric objects with the metric~$g$ explicitly, e.g.~$\Sigma_g M, \nabla^{s,g}$, $\D^g$,~$\Penr^g$ etc. 
  
Now let~$u\in C^\infty(M)$ and consider the conformal metric~$g_u=e^{2u}g$. 
The map~$b\colon X\mapsto e^{-u}X$ for~$1\le i\le n$ is an isometry between~$(TM,g)$ and~$(TM,e^{2u}g)$, which gives rise to an~$\SO(n)$-equivariant map~$b\colon P_{\SO}(M,g)\to P_{\SO}(M, e^{2u}g)$. 
By lifting~$b$ to the principal~$\Spin(n)$-bundles and then inducing it on the associated spinor bundles, we get an isometric isomorphism
\begin{equation}
 \beta\equiv \beta_{g,g_u}\colon (\Sigma_{g}M, g^s)\to (\Sigma_{g_u}M, g^s_u).
\end{equation}
The map~$\beta$ does not respect the spin connections: for~$X\in\Gamma(TM)$ and~$\psi\in \Gamma(\Sigma_g M)$,
\begin{align}
 \nabla^{s, g_u}_X \beta(\psi)-\beta(\nabla^{s,g}_X\psi)
 =-\frac{1}{2}\beta\left(\gamma^g(X)\gamma^g(\grad^g u)\psi+X(u)\psi\right)
\end{align}
Consequently, by a direct computation, we can get
\begin{align}
 \D^{g_u}\beta(\psi)=e^{-u}\beta\left(\D^g\psi+\frac{n-1}{2}\gamma^g(\grad^g u)\psi\right),
\end{align}
\begin{align}
 \Penr^{g_u}_X\beta(\psi)
 =\beta\left(\nabla^{s,g}_X\psi+\frac{1}{n}\gamma^g(X)\D^g\psi\right)
  -\frac{1}{2n}\beta\left(\gamma^g(X)\gamma^g(\grad^g u)\psi\right)
  -\frac{1}{2}X(u)\beta\left(\psi\right), \forall X\in\Gamma(TM).
 \end{align}
The non-homogeneous parts could be eliminated by introducing suitable weights:
\begin{align}\label{eq:conformaldirac}
 \D^{g_u}\beta(e^{-\frac{n-1}{2}u}\psi)
 = e^{-\frac{n+1}{2}u}\beta(\D^g\psi),
\end{align}
\begin{align}\label{eq:conformalpenrose}
 \Penr^{g_u} \beta\left(e^{\frac{u}{2}u}\psi\right)
 =e^{\frac{u}{2}}\beta(\Penr^g\psi).  
\end{align}

The summands in the functional~$\mathcal{L}$ are conformally invariant: setting $\varphi:=\beta(e^{-\frac{n-1}{2}u}\psi)$, there holds
\begin{equation}
 \int_M\left<\D^g\psi,\psi\right>_{g^s}\dv_g
 =\int_M\left<\D^{g_u}\varphi, \varphi\right>_{g^s_u}\dv_{g_u}, 
\end{equation}
\begin{equation}
 \int_M |\psi|^{2^\sharp}_{g^s}\dv_g
 =\int_M |\varphi|^{2^\sharp}_{g^s_u}\dv_{g_u}. 
\end{equation}
Consequently the action~$\cL$ in~\eqref{eq:action} (here we quote it for a general $M$) is conformally invariant, and hence also  equation~\eqref{eq:nld}. 

\subsection{Transformations induced by conformal diffeomorphisms}
\label{sect:conformal diffeomorphism}
Let~$f\colon M\to M$ be a diffeomorphism preserving the orientation and the spin structure~$\sigma$. 
Let~$g_f\equiv f^*g$ denote the pull-back metric on~$TM$, then the tangent map~$T f\colon (TM,g_f)\to (TM,g)$ is an isometry, hence it also preserves the Levi-Civita connections.
Since~$f$ is assumed to preserve the spin structure, we have an isomorphism~$\Spin(f)\colon P_{\Spin}(M,g_f)\to P_{\Spin}(M,g)$ which covers the equivariant morphism~$\SO(f)=Tf\colon P_{\SO}(M,g_f)\to P_{\SO}(M,g)$.  
Thus there is an induced map~$F$ which also covers the map~$f$ in the sense that the following diagram is commutative:
\begin{center}
 \begin{tikzcd}
  (\Sigma_{g_f}M, g_f^s) \arrow[r, "F"] \arrow[d]& (\Sigma_gM,g^s) \arrow[d] \\
  (M,g_f) \arrow[r, "f"] & (M,g) 
 \end{tikzcd}
\end{center}
The map~$F$ preserves the spin connection and is an isometry of vector bundles, hence also preserves the Dirac operators: for any~$\psi\in \Gamma(\Sigma_gM)$, write~$f^*\psi=F^{-1}\circ \psi\circ f\in \Gamma(\Sigma_{g_f}M)$ for the pull back spinor, then
\begin{equation}
 F\parenthesis{\D_{g_f}(f^*\psi)(x)}
 =\D_g \psi(f(x)),\qquad x\in M. 
\end{equation}

Suppose in addition that~$f$ is a conformal diffeomorphism, i.e.~$g_f=f^*g=e^{2u}g$ for some~$u\in C^\infty(M)$.
Then a solution~$\psi\in\Gamma(\Sigma_g M)$ of~\eqref{eq:nld} corresponds to another solution~$\varphi\in\Gamma(\Sigma_g M)$ via the following procedure:
\begin{equation}
 \begin{pmatrix}
  \psi\in \Gamma(\Sigma_g M)\\ \D_g\psi=|\psi|^{\frac{2}{n-1}}\psi
 \end{pmatrix}
 \mapsto 
 \begin{pmatrix}
  \psi_f\equiv f^*\psi\in\Gamma(\Sigma_{g_f} M) \\
  \D_{g_f}(\psi_f)=|\psi_f|^{\frac{2}{n-1}}\psi_f
 \end{pmatrix}
 \mapsto
 \begin{pmatrix}
  \varphi\coloneqq e^{\frac{n-1}{2}u}\beta^{-1}(\psi_f)\in\Gamma(\Sigma_g M) \\
  \D_g\varphi=|\varphi|^{\frac{2}{n-1}}\varphi
 \end{pmatrix}. 
\end{equation}

\noindent\textbf{Example.}
Let~$p:\mathbb{S}^{n}\backslash\{N\}\rightarrow \R^{n}$ be the stereographic projection, where $N\in\mathbb{S}^{n}$ is the north pole. 
Using the ambient coordinate~$\sph^n\subset \R^{n+1}(y^1,\cdots, y^{n+1})$ and~$\R^n(x^1,\cdots, x^n)$, we have 
\begin{equation}\label{eq:stereographic projection}
 \sph^n\setminus\{N\}\ni y\mapsto p(y)=x\in\R^n, \quad \mbox{ with }  x^i=\frac{2y^i}{1-y^{n+1}}.  
\end{equation}
The inverse of~$p$ will be denoted by~$\pi\colon \R^n\to\sph^n\setminus\{N\}\subset\R^{n+1}$, 
\begin{equation}\label{eq:inverse of stereographic projection}
 x\mapsto \pi(x)=y, \quad \mbox{ with } y^{i}=\frac{2x^i}{|x|^2+1},\; (1\le i\le n), \; y^{n+1}=\frac{|x|^2-1}{|x|^2+1}. 
\end{equation}
These are conformal maps, i.e. they satisfy 
\begin{align}\label{eq:stereographic projection is conformal}
 \pi^*g_0(x)=\parenthesis{\frac{2}{1+|x|^2}}^2 g_{\R^n}(x), & &
 p^*g_{\R^n}(y)=\parenthesis{\frac{1}{1-y^{n+1}}}^2 g_0. 
\end{align}
Now let~$\psi\in \mathring{H}^{\frac{1}{2}}(\R^n,\C^N)$ be a solution of~\eqref{eq:nld}, and 
set 
\begin{equation}
 \varphi\coloneqq \parenthesis{\frac{1}{1-y^{n+1}}}^{\frac{n-1}{2}}p^*\psi
 \in \Gamma(\Sigma_{g_0}\parenthesis{\sph^n\setminus\{N\}}). 
\end{equation}
Then~$\varphi$ is a solution to 
\begin{equation}\label{nldpuncturedsphere}
\D_{g_{\mathbb{S}^{n}}}\varphi=\vert\varphi\vert^{2^{\sharp}-2}\varphi\qquad\mbox{on}\;\mathbb{S}^{n}\backslash\{N\}
\end{equation}
and $\int_{\mathbb{S}^{n}}\vert\varphi\vert^{2^{\sharp}}\;d\vol_{g_{\mathbb{S}^{n}}}<\infty$. 
This allows to prove (see \cite{Ammann}) that $\varphi$ extends  to a weak solution on $\mathbb{S}^{n}$. 
Thus there exists a one-to-one correspondence between weak solutions to \eqref{eq:nld} in $\mathring{H}^{\frac{1}{2}}(\R^{n},\Sigma\R^{n})$ and weak solutions to \eqref{nldpuncturedsphere} in the space $H^{\frac{1}{2}}(\mathbb{S}^{n},\Sigma\mathbb{S}^{n})$.

Recall that the~$-\frac{1}{2}$-Killing spinors on~$(\sph^n,g_0)$ have suitable constant length and are eigenspinors, thus being particular solutions of~\eqref{eq:nld}, as recalled in Proposition \ref{prop:killing-on-spheres}. Moreover, they are also ground states as they verify \eqref{eq:gsdef}.
 The M\"obius group of conformal diffeomorphism of the round sphere $\sph^n$ can be expressed in terms of $\R^n$ via stereographic projection. Then, exploiting the conformal invariance of \eqref{eq:nld} Killing spinors on the sphere can be mapped to a family of solutions of~\eqref{eq:nld}. 
Thus it is natural to ask whether such family spinors are the only ground states of $\eqref{eq:action}$. Our aim is to give a positive answer to this question, also providing an explicit formula for the minimizers. This is the content of Theorem \ref{thm:main} and of Corollary \ref{cor:bubbles on Rn}


\subsection{Estimates on conformal eigenvalues}
\label{sect:estimate of conformal eigenvalues}
For  convenience of the reader, in this section we briefly recall the proof of the lower bound \eqref{eq:lowerbound}.
Let~$(M,g)$ be a closed spin manifold and consider the conformal class of~$g$
\begin{equation}
 [g]=\left\{g_u\coloneqq e^{2u}g\mid u\in C^\infty(M)\right\}.
\end{equation}
The dimension of harmonic spinors~$\dim_\C \ker(\D^g)=\dim_\C\ker(\D^{g_u})\ge 0$ is a conformal invariant, as it easily follows from \eqref{eq:conformaldirac}.
Let~$\lambda_{1}(\D^g)$ be the first positive eigenvalue of~$\D^g$, then~$\lambda_1(\D^{g_u})$ depends on~$u$ continuously and never vanishes for $g_u\in[g]$.

In~\cite{ammannlowerbound}, B. Ammann considered the following quantity (using a different notation, also noting that we fixed the spin structure throughout)  
\begin{equation}
 \Lambda_{1}(M,[g])\coloneqq \inf_{g_u\in [g]}\left\{\lambda_1(\D^{g_u})\Vol(g_u)^{\frac{1}{n}}\right\}.
\end{equation}
Let~$\mathcal{C}$ be the orthogonal complement of~$\ker(\D^g)$, so~$L^2(M;\Sigma M)=\ker(\D^g)\oplus \mathcal{C}$; and let~$\mathcal{C}^*$ be the set of non-zero elements in~$\mathcal{C}$. 
\begin{lemma}[{\cite[Prop. 2.4.]{ammannlowerbound}}]
 $\Lambda_1(M,[g])=\inf_{\varphi\in\mathcal{C}^*}\left\{\frac{\|\varphi\|_{L^{\frac{2n}{n+1}}}^2}{\int_M \left<\varphi,|\D^g|^{-1}\varphi\right>\dv_g}\right\}$.
\end{lemma}
By taking~$\phi=(\D^{g})^{-1}(\varphi)$ and~$\phi\perp \ker(\D^g)$, one can see that
 \begin{equation}
  \Lambda_1(M,[g])=\inf_{\phi\in W^{1,\frac{2n}{n+1}}(\Sigma_g M)\cap \ker(\D^g)^\perp, \\ \phi\neq 0} 
  \frac{\left(\int_M |\D^g\phi|^{\frac{2n}{n+1}}\dv_g\right)^{\frac{n+1}{n}}}{\int_M \left<\D^g\phi,\phi\right>\dv_g} 
 \end{equation}
Denote the standard round metric on~$\sph^n\subset \R^{n+1}$ by~$g_0$, then~$\lambda_1(\D^{g_0})=\frac{n}{2}$ and~$\Vol(g_0)=\omega_n$.  
Similarly to  the Yamabe problem, we have that~$\Lambda_1$ attains its maximum on the standard unit sphere.
\begin{lemma}[{\cite[Thm. 1.1]{Ammann-Grosjen-Humbert-Morel08}}]\label{lemma:conformal first eigenvalue bounded by sphere}
 $\Lambda_1(M,[g])\le \Lambda_1(\sph^n,[g_0])=\frac{n}{2}\omega_n^{\frac{1}{n}}$.
\end{lemma}

As proved by T. Isobe in \cite{Isobecritical}, if ~$\psi$ is a nontrivial solution of~\eqref{eq:nld}, then
\begin{align}
 \Lambda_1(\sph^n,g_0)=\frac{n}{2}\omega_n^{\frac{1}{n}}
 \le\frac{\left(\int_{\sph^n} |\D^{g_0}\psi|^{\frac{2n}{n+1}}\dv_{g_0}\right)^{\frac{n+1}{n}}}{\int_{\sph^n} \left<\D^{g_0}\psi,\psi\right>\dv_{g_0}}
 =&\frac{\left(\int_{\sph^n} |\psi|^{\frac{2n}{n-1}}\dv_{g_0}\right)^{\frac{n+1}{n}}}{\int_{\sph^n} |\psi|^{\frac{2n}{n-1}}\dv_{g_0}} \\
 =&\left(\int_{\sph^n}|\psi|^{2^*}\dv_{g_0}\right)^{\frac{1}{n}}. 
\end{align}
Hence  
\begin{equation}
 \mathcal{L}(\psi)=\frac{1}{2n}\int_{\sph^n} |\psi|^{2^*}\dv_{g_0} 
 \ge \frac{1}{2n}\left(\frac{n}{2}\right)^n\omega_n, 
\end{equation}
which shows the lower bound of the non-trivial critical levels in \eqref{eq:lowerbound}. 

\subsection{Capacity and Sobolev spaces}
The concept of capacity is useful in regularity theory, and we recall some basic facts here, which can be found in~\cite{saari,ZiemerGTM120}.       
 
Let~$\Omega\subseteq \R^n$ be a connected open domain and~$K\Subset \Omega$ a compact subset. 
Let~$p>1$ be a fixed number.
The set of admissible potentials are
\begin{equation}
 W_0(K,\Omega)\coloneqq\{u\in W_0^{1,p}(\Omega)\cap C^0(\Omega)\mid u\ge \mathds{1}_K\}
\end{equation}
where~$\mathds{1}_K$ is the characteristic function of~$K$. 
The~$p$-capacity of~$(K,\Omega)$ is defined as
\begin{align}
 \capacity_p(K,\Omega)\coloneqq
 \inf_{u\in W_0(K,\Omega)} \int_\Omega |\nabla u|^p\dd x.
\end{align}
Then, we can also define the~$p$-capacity of an open subset~$U\subset \Omega$ via inner exhaustion by compact subsets, and then the~$p$-capacity of a general measurable set~$E\subset \Omega$ via outer approximation by open neighborhoods, see~\cite{saari}.  
\begin{defi}
 A set~$E\subset\R^n$ is said to be of ~$p$-capacity zero if~$\capacity_p(E,\Omega)=0$ for all open sets~$\Omega\subset\R^n$. 
 (Equivalently,~$\capacity_p(E,\Omega_0)=0$ for some~$\Omega_0\subset\R^n$.)
\end{defi}
The capacity of a set is closely related to its Hausdorff measure~$\mathcal{H}$.
\begin{proposition}\label{prop:capacity and Hausdorff measure}
 Let~$E\subset\R^n$ and~$1<p\le n$. Then the following implications hold. 
 \begin{enumerate}
  \item[(i)] If~$\capacity_p(E)=0$, then~$\dim_{\mathcal{H}}(E)\le n-p$. 
  \item[(ii)] If~$\mathcal{H}^{n-p}{}(E)<\infty$, then~$\capacity_p(E)=0$.
 \end{enumerate}
\end{proposition}
Functions in the space~$W_0^{1,p}(\Omega)$ cannot see the sets of~$p$-capacity zero. More precisely, we have
\begin{proposition}\label{prop:capacity and negligibility}
 Let~$\Omega\subset\R^n$ be open and~$E\subset\Omega$ relatively closed.
 Then~$W^{1,p}_0(\Omega)=W^{1,p}_0(\Omega\backslash E)$ iff~$\capacity_p(E)=0$.  
\end{proposition}

\section{Classification in dimension two}\label{sec:twodimension}
Roughly speaking, the strategy of the proof of the main result, Theorem \ref{thm:main}, consists of using the modulus of the spinor to make a conformal change of the metric on $\sph^n$, which allows then to use some rigidity result to conclude the claim. This method is similar to the one used by Ammann in \cite{ammannsmallest}, see also Remark \ref{rmk:ammannclassification}.

In dimension two the equation has a smooth structure, so the nodal set is already known to be discrete by the result in~\cite{bar1999zero}. In this case, after the conformal change of the metric, we can use the classification result associated with eigenvalue estimates, due to B\"ar \cite{bar0}. To this aim, we exploit the fact that ground state solutions of~\eqref{eq:nld} on~$\sph^2$ do not admit zeros.
\begin{proposition}[{\cite[Prop. 3.7]{brandingnodal}}] \label{prop:nozeros}
 Let~$\psi\in\Gamma(\Sigma M)$ with $\|\psi\|_{L^4}=1$ be a solution of
 \begin{equation}
  \D\psi=\mu|\psi|^2\psi, 
 \end{equation}
 where~$\mu\in\R$. 
 Let~$N(\psi)$ denote the sum of orders of zeros of~$\psi$. 
 Then 
 \begin{equation}
  N(\psi)\le \frac{\mu^2}{4\pi}-\frac{\chi(M)}{2}.
 \end{equation}
 \end{proposition}
For a ground state solution,~$\mu=1$ and~$\chi(\sph^2)=2$. 
It follows that~$\psi$ never vanishes. 
Moreover, by elliptic regularity theory one can prove that~$\psi\in C^{\infty}(\sph^2,\Sigma\sph^2)$, see e.g.~\cite{borrellifrank, jost2018regularity}.

\begin{thm}\label{thm:2Dsolution}
Let $\psi\in H^{1/2}(\sph^2,\Sigma_{g_0}\sph^2)$ be a ground state solution to \eqref{eq:nlds} with $n=2$.
Then,~$\psi$ is a~$(-\frac{1}{2})$-Killing spinor up to a conformal diffeomorphism. 
More precisely, there exist a ~$(-\frac{1}{2})$-Killing spinor~$\Psi\in\Gamma(\Sigma_{g_0}\sph^2)$ and a conformal diffeomorphism~$f\in \Conf(\sph^2,g_0)$ such that
\begin{equation} 
 \psi=\parenthesis{\det(\dd f)}^{\frac{1}{4}}\beta^{-1}(f^*\Psi), 
\end{equation}
where~$\beta\colon \Sigma_{g_0}\sph^n\to \Sigma_{f^*g_0}\sph^n$ is the conformal identification of the spinor bundles.  
\end{thm}
\begin{proof}
Let $\psi\in H^{1/2}(\sph^2,\Sigma_{g_0}\sph^2)$ be a ground state solution to \eqref{eq:nlds}.
We know that~$\psi$ is smooth and has no zeros (see Proposition \ref{prop:nozeros}).
Consider the conformal metric 
\be\label{eq:newmetric}
\overline{g}=\vert\psi\vert^{4}g_0\,,
\ee
with volume
\be\label{eq:newvolume}
\vol_{\overline{g}}(\sph^2)=\int_{\sph^2}\vert\psi\vert^4\dd{\vol_{g_0}}=4\pi=\omega_2\,,
\ee
since~$\psi$ is a ground state solution. 
Let~$\beta\colon \Sigma_{g_0}\sph^2\to \Sigma_{\overline{g}}\sph^2$ denote the corresponding isomorphism of the associated spinor bundles, and define
\be\label{eq:newspinor}
\phi=\frac{\beta(\psi)}{\vert\psi\vert}\,, 
\ee
which has constant length $|\phi|\equiv 1$.  

The conformal invariance of \eqref{eq:nlds} implies that $\phi$ solves the same equation
\be\label{eq:newequation}
\D^{\overline{g}}\phi=\vert\phi\vert^{2}\phi=\phi\,,\qquad\mbox{on $(\sph^2,\overline{g})$}\,.
\ee
In \cite{bar0}, the following lower bound for the eigenvalues of the Dirac operators on a closed surfaces $(M^2,g)$ was proved
\be\label{eq:lbeuler}
\lambda^2\geq\frac{2\pi\chi(M^2)}{\vol_g(M^2)}\,,
\ee
where~$\chi(M^2)$ is the Euler characteristic of the surface.
Moreover, equality is attained if and only if the surface is isometric to the round sphere $\sph^2$, or to the torus $\mathbb{T}^2$ with the flat metric and the trivial spin structure. 

Let~$f\colon (\sph^2,\overline{g})\to (\sph^2, g_0)$ be the isometry given above. 
It can be used to transform the spinor~$\phi$ to a spinor on the round sphere~$(\sph^2,g_0)$. 
More precisely, the isometry~$f$ induces an isomorphism~$F\colon\Sigma_{\overline{g}}\sph^2 \to \Sigma_{g_0}\sph^2$ which preserves the Hermitian structures and the spin connections, hence also the Dirac operators. 
Then the induced spinor is
\begin{equation}
 \phi_f\coloneqq (f^{-1})^*\phi= F\circ \phi\circ f^{-1}\in \Gamma(\Sigma_{g_0}\sph^2)
\end{equation}
which has the same properties as~$\phi$, namely
\begin{align} 
 |\phi_f|\equiv 1, & & \D^{g_0}\phi_f=\phi_f, & & \Penr^{g_0}\phi_f=0.
\end{align}
In particular, $\phi_f$ is a~$(-\frac{1}{2})$-Killing spinor, say $\phi_f=\Psi\in\Killing(g_0;-\frac{1}{2})$. 
Then ~$\phi=f^*\Psi=F^{-1}\circ \Psi\circ f\in \Gamma(\Sigma_{\overline{g}}\sph^2)$. 

Note that the isometry~$f$ is also conformal:~$f^*g_0= \overline{g}=|\psi|^4 g_0$, and~$\psi$ can be obtained by the induced conformal transformations on spinors from~$\phi$, namely
\begin{equation}
 \psi=\det\parenthesis{\dd f}^{\frac{1}{4}}\beta^{-1}\parenthesis{f^*\Psi}\in \Gamma(\Sigma_{g_0}\sph^2).
\end{equation}
This concludes the proof.  
\end{proof}

Though being elegant, the proof above is not constructive enough, and the argument does not directly generalize to higher dimensions.
The reason, among others, is due to the lack of a strong rigidity statement in the eigenvalues estimate: we do not know whether the round metric is the only (up to isometry) metric assuming the extremals of the first positive conformal eigenvalue or not (see Section \ref{sect:estimate of conformal eigenvalues}). 
In the following we take a closer look at the curvatures of the conformally related metrics~$g_0$ and~$\overline{g}$. 
We will see that the length function~$|\psi|\colon \sph^2\to \R$ actually determines the conformal isometry~$f$ (up to rigid motions) and vice versa.  
This idea continues to work in general dimensions. 

From the pointwise Lichnerowicz formula (see e.g. \cite[Theorem 3.4.1]{Jost})
\begin{align}\label{eq:Bochner-formula}
 \frac{1}{4}\Scal_{\overline{g}}\phi
 =\parenthesis{\D^{\overline{g}}}^2\phi-\parenthesis{\nabla^{s,\overline{g}}}^*\parenthesis{\nabla^{s,\overline{g}}}\phi  
\end{align}
we get the integral Bochner--Lichnerowicz formula
\begin{align}
 \int_{\sph^2}|\D^{\overline{g}}\phi|^2\dv_{\overline{g}}
 =\int_{\sph^2}|\nabla^{s,\overline{g}}\phi|^2\dv_{\overline{g}}
  +\frac{1}{4}\int_{\sph^2}\Scal_{\overline{g}} |\phi|^2\dv_{\overline{g}},
\end{align}
where~$\Scal_{\overline{g}}=2K_{\overline{g}}$ denotes the scalar curvature of~$\overline{g}$. 
Substituting~\eqref{eq:Penrose-Dirac decomposition} into it, we obtain 
\begin{align}
 \int_{\sph^2}|\Penr^{\overline{g}}\phi|^2\dv_{\overline{g}} 
 =&\frac{1}{2} \int_{\sph^2}(1-K_{\overline{g}})|\phi|^2\dv_{\overline{g}}
 =\frac{1}{2} \int_{\sph^2}(1-K_{\overline{g}})\dv_{\overline{g}}\\
 =&\Vol_{\overline{g}}(\sph^2)-2\pi\chi(\sph^2)=0\,.
\end{align}
Hence~$\Penr^{\overline{g}}\phi=0$, i.e. ~$\phi$ is a twistor spinor on~$(\sph^2,\overline{g})$.  
It follows that~$\phi$ is a~$-\frac{1}{2}$-Killing spinor and \eqref{eq:newequation} and \eqref{eq:Bochner-formula} give
\begin{align}
 \frac{1}{2}K_{\overline{g}}\phi
 =\phi- \frac{1}{2}\phi=\frac{1}{2}\phi.
\end{align}
Since~$\phi$ is nowhere vanishing, we conclude that~$K_{\overline{g}}\equiv 1$. 

Thus the conformal factor~$u=\log|\psi|^2$ satisfies~$\overline{g}=e^{2u}g_0$ and solves the equation
\begin{align}\label{eq:prescribing curvature equation-dim two}
 -\Delta_{g_0}u+ K_{g_0}=K_{\overline{g}} e^{2u}, 
 & & 
 \mbox{i.e. }  
 -\Delta_{g_0} u+1=e^{2u}.
\end{align}
It is well-known that the solutions of \eqref{eq:prescribing curvature equation-dim two} have the form
\begin{equation}
 u= \frac{1}{2}\log \det (\dd f), 
\end{equation}
with~$f\in\Conf(\sph^2, g_0)$ being a conformal trasformation:~$f^* g_0=e^{2u}g_0=|\psi|^4g_0$.
Thus~$f\colon (\sph^2,\overline{g})\to (\sph^2,g_0)$ is an isometry, which is the one in our proof, up to rigid motions.

\begin{remark}
 The length function~$|\psi|$ can be explicitly given. 
 Indeed, fixing~$K_{\overline{g}}\equiv 1$ and noting that  equation~\eqref{eq:prescribing curvature equation-dim two} is conformally invariant, we can use the stereographic projection 
 \begin{align}
  \pi\colon \R^2\ni z\mapsto y=\parenthesis{\frac{2Re(z)}{1+|z|^2},\frac{2Im(z)}{1+|z|^2},\frac{-1+|z|^2}{1+|z|^2}}\in\sph^2\subset\R^3 
 \end{align} 
 to pull the equation back to~$\R^2$.
 Since~$\pi\colon\R^2\to\sph^2$ is conformal~\eqref{eq:stereographic projection is conformal}, 
 the function
 \begin{equation}
 v \coloneqq u\circ \pi(z)+\ln\parenthesis{\frac{2}{1+|z|^2}}
 \end{equation}
 is a solution of
 \begin{equation}
  -\Delta_{\R^2} v=e^{2v}  \quad \mbox{ in } \R^2, 
 \end{equation}
 and by  conformal invariance
 \begin{equation}
  \int_{\R^2} e^{2v}\dd x
  = \int_{\sph^2} e^{2u}\dv_{g_0}<\infty, 
 \end{equation}
 since~$u\in C^\infty(\sph^2)$.   
 Such solutions~$v$ were classified, see e.g.~\cite{ChenLi1991classification}: there exist~$\lambda>0$ and~$z_0\in\R^2$ such that
 \begin{equation} 
  v(z)=\frac{1}{2}\ln\parenthesis{\frac{32\lambda^2}{(4+\lambda^2|z-z_0|^2)^2}}-\frac{1}{2}\ln 2. 
 \end{equation} 
Since~$u=\ln|\psi|^2$, we see that the length function of the spinor~$\psi$ is given by
\begin{equation}
 |\psi(y)|=\parenthesis{\frac{2\lambda(1+|p(y)|^2)}{4+\lambda^2|p(y)-z_0|^2}}^{\frac{1}{2}}.   
\end{equation}
\end{remark}
 
\section{Classification in higher dimensions}\label{sec:higherdimension}
The proof of Theorem \ref{thm:main} for $n\geq3$ essentially relies upon the same ideas as in the two dimensional case.
However, the higher-dimensional case is technically more delicate, since we have less information on the nodal set of the spinor, and thus the solution is a priori only of class~$C^{1,\alpha}$. 
The proof of Theorem \ref{thm:main} for $n \geq 3$ requires to estimate  the Hausdorff dimension of the nodal set \eqref{eq:nodalset}, see Theorem~\ref{thm:nodalset}.   
We postpone its proof and present it in the next section in order to simplify the proof of Theorem \ref{thm:main3}.
 
We start by estimating the perimeter of the boundary of the tubular neighborhoods for a set of Hausdorff dimension less than~$n-1$.  
We denote by $\cH^{s}(\cdot)$ the $s$-dimensional Hausdorff measure.

\begin{lemma}\label{lemma:Zeps}
Let~$Z\subset\R^n$ be a compact $(n-1)$-rectifiable set with~$\cH^{n-1}(Z)=0$.
For any $\eps>0$, consider the~$\eps$-tubular neighborhood $Z_\eps:=\{x\in\sph^n\,:\,\dist(x,Z)\leq\eps \}$. 
Then for a.e.~$\eps>0$, the boundary~$\p Z_\eps$ is ~$(n-1)$-rectifiable and  along a sequence~$\eps_k\to 0^+$,  
\be\label{eq:Z boundary limit}
\lim_{k\to\infty}\cH^{n-1}(\partial Z_{\eps_k})=0\,.
\ee
\end{lemma}  
 \begin{proof}
 It is well-known that there exists~$\eps_0>0$ such that for all but countably many~$\eps\in (0,\eps_0)$ the set~$\p Z_\eps$ is~$(n-1)$-rectifiable, see e.g.~\cite[Section 5]{KLV13}.
 Moreover, applying \cite[Prop. 5.8]{KLV13} with $\lambda=n-1$, we get
 \be\label{eq:n-1estimate}
 \cH^{n-1}(\partial Z_\eps)\leq C(n)\mathcal{M}^{n-1}_\eps(Z)\,,
 \ee
 where $\mathcal{M}^{n-1}_\eps$ is the $(n-1)$-dimensional Minskowski $\eps$-content \cite[\S 4]{mattila}, and the dimensional constant~$C(n)$ is independent of~$\eps$.
 Now~$Z$ is compact and~$(n-1)$-rectifiable with~$\cH^{n-1}(Z)=0$, so its $(n-1)$-Minkowski content is well-defined and coincides with the $(n-1)$-Hausdorff measure \cite[Theorem 3.2.39]{federer}. Then we have  
 \[
 \lim_{\eps\to0^+}\mathcal{M}^{n-1}_\eps(Z)
 =\mathcal{M}^{n-1}(Z)
 =\cH^{n-1}(Z)=0\,.
 \]
 and the claim follows by \eqref{eq:n-1estimate}.
 \end{proof}
 
We will apply Lemma~\ref{lemma:Zeps} to the nodal set~$\zero(\psi)$ of a solution~$\psi$ to~\eqref{eq:nlds}.
Since~$\psi$ has regularity~$C^{1,\alpha}$, its zero set~$\zero=\zero(\psi)$ is  closed  in the compact space~$\sph^n$, hence it is also compact. 
By Theorem~\ref{thm:nodalset}, it has Hausdorff dimension at most~$(n-2)$, in particular~$\cH^{n-1}(\zero)=0$.   
Note that~$\sph^n\backslash \zero$ is a non-empty open set of full measure. 
Thus, up to a stereographic projection,~$\zero$ can be viewed as a compact subset of~$B_R(0)\subset \R^n$ for some~$R<\infty$.
Since the Hausdorff measures on~$B_R(0)$ with respect to the Euclidean metric and the conformal spherical metric are uniformly equivalent, we can apply Lemma~\ref{lemma:Zeps} to conclude that, along a sequence~$\eps_k\to 0^+$,
\be\label{eq:Zepslimit}
\lim_{\eps_k\to0^+}\cH^{n-1}(\partial \zero_{\eps_k})=0\,.
\ee
 
\begin{thm}\label{thm:main3}
Let $\psi\in H^{1/2}(\sph^n,\Sigma_{g_0}\sph^n)$ be a ground state solution to \eqref{eq:nlds} with $n\geq3$. 
Then,~$\psi$ is a~$(-\frac{1}{2})$-Killing spinor up to a conformal diffeomorphism. 
More precisely, there exist a ~$(-\frac{1}{2})$-Killing spinor~$\Psi\in\Gamma(\Sigma_{g_0}\sph^n)$ and a conformal diffeomorphism~$f\in \Conf(\sph^n,g_0)$ such that
\begin{equation} 
 \psi=\parenthesis{\det(\dd f)}^{\frac{n-1}{2n}}\beta^{-1}(f^*\Psi), 
\end{equation}
 where~$\beta\colon\Sigma_{g_0}\sph^n\to\Sigma_{f^*g_0}\sph^n$ is the conformal identification.  
\end{thm}

\begin{proof}
Let $\psi\in H^{1/2}(\sph^n,\Sigma_{g_0}\sph^n)$ be a ground state solution to \eqref{eq:nlds}, with $n\geq3$.
Consider the conformal change of metric on $\sph^n\setminus \zero$
\be\label{eq:confchange}
\overline{g}
=\parenthesis{\frac{2}{n}}^{2}\vert\psi\vert^{4/(n-1)}g_0\,.
\ee
Notice that the total volume is preserved 
\be\label{eq:volume}
\vol_{\overline{g}}(\sph^n\setminus \zero)
=\parenthesis{\frac{2}{n}}^n\int_{\sph^n}\vert\psi\vert^{\frac{2n}{n-1}}\dd{\vol_{g_0}}
=\omega_n=\Vol_{g_0}(\sph^n)\,,
\ee
since $\psi$ is a ground state solution and $\cL_{g_0}(\psi)=\frac{1}{2n}\int_{\sph^2}\vert\psi\vert^{\frac{2n}{n-1}}\dd{\vol_{g_0}}$\, (see \eqref{eq:gsdef}).

As before, let~$\beta=\beta_{g_0,\overline{g}}\colon \Sigma_{g_0} (\sph^n\setminus \zero) \to \Sigma_{\overline{g}}(\sph^n\setminus \zero)$ be the isometry associated to the conformal change of the metric and define the spinor
\be\label{eq:confspinor}
\phi=\parenthesis{\frac{n}{2}}^{\frac{n-1}{2}}\frac{\beta(\psi)}{\vert\psi\vert}\,,
\qquad 
\vert\phi\vert\equiv\parenthesis{\frac{n}{2}}^{\frac{n-1}{2}}\,.
\ee
Denote the nodal set of~$\psi$ by~$\zero=\zero(\psi)$, then  
\be\label{eq:boundedness}
\phi\in C^{\infty}(\sph^n\setminus \zero)\cap L^\infty(\sph^n\setminus \zero)\,.
\ee 
Note that~$\phi$ is an eigenspinor for the $\D^{\overline{g}}$-Dirac operator, i.e. 
\be\label{eq:confequation}
\D^{\overline{g}}\phi
=\vert\phi\vert^{2^\sharp-2}\phi
=\frac{n}{2}\phi\,,\qquad\mbox{on $(\sph^n\setminus \zero,\overline{g})$}\,
\ee
in the classical sense. 

Fix $\eps>0$ and consider the neighborhood $\zero_\eps$ of the nodal set as in Lemma \ref{lemma:Zeps}. Observe that the metric $\overline{g}$ is regular and Riemannian on $\sph^n\setminus\zero_\eps$, so here we can consider the pointwise Bochner--Lichnerowicz formula \cite[Theorem 3.4.1]{Jost}
\be\label{eq:Bochner-Lichnerowicz pointwise}
\parenthesis{\D^{\overline{g}}}^2 
=\parenthesis{\nabla^{s,\overline{g}}}^*\parenthesis{\nabla^{s,\overline{g}}}
+\frac{\Scal_{\overline{g}}}{4}\,,
\ee  
where $\Scal_{\overline{g}}$ is the scalar curvature of the metric $\overline{g}$.
It follows that
\begin{align}\label{eq:Bochner-Lichnerowicz pointwise-integral}
 \int_{\sph^n\setminus\zero_\eps}
    \langle(\D^{\overline{g}})^2\phi,\phi\rangle 
    \dd{\vol}_{\overline{g}}
=\int_{\sph^n\setminus\zero_\eps}
    \langle\nabla^{s,\overline{g}*}\nabla^{s,\overline{g}}\phi,\phi\rangle\dd{\vol}_{\overline{g}} 
    +\int_{\sph^n\setminus\zero_\eps}
        \frac{\Scal_{\overline{g}}}{4}\vert\phi\vert^2\dv_{\overline{g}}\,.
\end{align}
We claim that the integral form of Bochner--Lichnerowicz's formula
\begin{align}\label{eq:Bochner-Lichnerowicz-integral}
 \int_{\sph^n\setminus\zero_\eps}
  \vert\D^{\overline{g}}\phi\vert^2 \dd{\vol}_{\overline{g}}
 =\int_{\sph^n\setminus\zero_\eps}
    \vert\nabla^{s,\overline{g}}\phi\vert^2
        \dd{\vol}_{\overline{g}}
 +\int_{\sph^n\setminus\zero_\eps}
    \frac{\Scal_{\overline{g}}}{4}\vert\phi\vert^2
        \dd{\vol}_{\overline{g}}
\end{align}
holds. 
Generally speaking, on a manifold with non-empty boundary, from~\eqref{eq:Bochner-Lichnerowicz pointwise-integral} one gets additional boundary integrals in \eqref{eq:Bochner-Lichnerowicz-integral}.
However, in our case,
\begin{equation}
 \Abracket{(\D^{\overline{g}})^2\phi,\phi}
 =\parenthesis{\frac{n}{2}}^2\Abracket{\phi,\phi}
 =\Abracket{\D^{\overline{g}}\phi,\D^{\overline{g}}\phi}, 
\end{equation}
\begin{align}
 \int_{\sph^n\setminus\zero_\eps} 
    \langle\nabla^{s,\overline{g}*}\nabla^{s,\overline{g}}\phi,\phi\rangle\dv_{\overline{g}} 
 =\int_{\sph^n\setminus\zero_\eps}\vert\nabla^{s,\overline{g}}\phi\vert^2\dd{\vol}_{\overline{g}} 
 -\int_{\partial\zero_\eps}\langle\nabla^{s,\overline{g}}_{\nu}\phi,\phi\rangle\dd{\cH}^{n-1}, 
\end{align}
and
\begin{align}
 2\Re\left<\nabla^{s,\overline{g}}_{\nu}\phi,\phi\right>
 =\p_\nu |\phi|^2=0,
\end{align}
whence \eqref{eq:Bochner-Lichnerowicz-integral}.  
Furthermore, the decomposition~\eqref{eq:Penrose-Dirac decomposition} and the eigenspinor equation~\eqref{eq:confequation} give
\begin{align}
 \parenthesis{\frac{n-1}{n}}\parenthesis{\frac{n}{2}}^2 
 \int_{\sph^n\setminus\zero_\eps}|\phi|^2\dv_{\overline{g}}
 =\int_{\sph^n\setminus\zero_\eps}\vert\Penr^{\overline{g}}\phi\vert^2\dv_{\overline{g}}
 +\frac{1}{4}\int_{\sph^n\setminus\zero_\eps}
    \Scal_{\overline{g}}|\phi|^2\dv_{\overline{g}}.
\end{align}
Since~$|\phi|$ is constant (see~\eqref{eq:confspinor}), it follows that
\begin{align}\label{eq:int-Penrose}
  \parenthesis{\frac{n-1}{n}}\parenthesis{\frac{n}{2}}^{n+1} 
 \Vol_{\overline{g}}(\sph^n\setminus\zero_\eps)
 =\int_{\sph^n\setminus\zero_\eps}\vert\Penr^{\overline{g}}\phi\vert^2\dv_{\overline{g}}
 +\frac{1}{4}\parenthesis{\frac{n}{2}}^{n-1}\int_{\sph^n\setminus\zero_\eps}
    \Scal_{\overline{g}}\dv_{\overline{g}}. 
\end{align}
In particular, this implies
\begin{equation}
 \int_{\sph^n\setminus\zero_\eps}\Scal_{\overline{g}}\dv_{\overline{g}}
 \le n(n-1)\Vol_{\overline{g}}(\sph^n\setminus\zero_\eps)
 \le n(n-1)\omega_n. 
\end{equation}

We need to control the integral of the new curvature~$\Scal_{\overline{g}}$ over the set~$\sph^n\setminus\zero_\eps$.
In dimension two this could be estimated using the Gauss--Bonnet formula, while in higher dimensions we use the Yamabe invariant. 

Recall that the Yamabe invariant of the conformal class~$[g_0]$ is defined as
\begin{equation}
 \YM(\sph^n,[g_0])=\min\braces{\frac{\int_{\sph^n}\Scal_{g}\dv_g}{\parenthesis{\int_{\sph^n}\dv_g}^{\frac{n-2}{n}}}\mid g\in [g_0]}, 
\end{equation}
where~$[g_0]$ denotes the conformal class of the round metric~$g_0$, which is equivalently characterized as 
\begin{equation}
 \YM(\sph^n,[g_0])
 =\min\braces{
 Q(u)\equiv
 \frac{\int_{\sph^n}c_n|\nabla^{g_0}u|^2+\Scal_{g_0}u^2\dv_{g_0}}{\parenthesis{\int_{\sph^n}u^{\frac{2n}{n-2}}\dv}^{\frac{n-2}{n}}}
 \mid
 u\in C^\infty(\sph^2), \; u>0 
 }. 
\end{equation}
By further taking the~$W^{1,2}$-closure of~$C^\infty(\sph^n)$, we get 
\begin{equation}\label{eq:yamabe quotient}
 \YM(\sph^n,[g_0])
 =\min\braces{Q(u)=
 \frac{\int_{\sph^n}c_n|\nabla^{g_0}u|^2+\Scal_{g_0}u^2\dv_{g_0}}{\parenthesis{\int_{\sph^n}|u|^{\frac{2n}{n-2}}\dv}^{\frac{n-2}{n}}}
 \mid
 u\in W^{1,2}(\sph^n), u\neq 0 
 }. 
\end{equation}
The value of the Yamabe invariant of the round sphere is well-known, and it is given by
\begin{equation}
 \YM(\sph^n,[g_0])=n(n-1)\omega_n^{\frac{2}{n}}. 
\end{equation}
 The metric $\overline{g}$ in \eqref{eq:confchange} might not be a Riemannian metric on~$\sph^n$ since~$\zero=\zero(\psi)$ might be non-empty, that is, the conformal factor might vanish at some points.

Define 
\begin{equation}\label{eq:hdefinition}
 h\coloneqq\parenthesis{\frac{2}{n}}^{\frac{n-2}{2}}\vert\psi\vert^{\frac{n-2}{n-1} }\in C^{0}(\sph^n)\cap C^{\infty}(\sph^n\setminus\zero)\,,
\end{equation}
then~$\overline{g}=h^{\frac{4}{n-2}}g_0$, and the scalar curvatures of the two metrics on~$\sph^n\setminus\zero_\eps$ are related by
\begin{equation}\label{eq:Yamabe}
 L_{g_0}h\equiv
 -c_n\Delta_{g_0} h+\Scal_{g_0}h
 =\Scal_{\overline{g}}h^{\frac{n+2}{n-2}}, 
 \qquad \mbox{ on } \sph^n\setminus\zero,  
\end{equation}
with~$c_n= 4\frac{n-1}{n-2}$ and~$\Scal_{g_0}=n(n-1)$. 
We need the following regularity result on~$h$. 

    \begin{lemma}\label{lemma:regularity of h} 
      With respect to the round metric~$g_0$, one has that $h\in H^1(\sph^{\blu{n}})$. 
    \end{lemma}
    \begin{proof}[Proof of Lemma~\ref{lemma:regularity of h}]
        Choose~$\eps>0$ small enough such that~$\p\zero_\eps$ is $(n-1)$-rectifiable, see the proof of Lemma \ref{lemma:Zeps}.
        Using~\eqref{eq:Yamabe}, an integration by parts gives
        \begin{align}\label{eq:estimate of gradient h}
         c_n\int_{\sph^n\setminus\zero_\eps}
            &\vert\nabla^{g_0}h\vert^2  \dd{\vol}_{g_0}
        =\int_{\partial\zero_\eps}c_n\frac{\partial h}{\partial \nu}h \dd{\vol}_{g_0}
        -\int_{\sph^n\setminus\zero_\eps}c_n(\Delta_{g_0} h)h \dd{\vol}_{g_0} \\
        =&\int_{\partial\zero_\eps}c_n\frac{\partial h}{\partial \nu}h \dd{\vol}_{g_0}
        -\int_{\sph^n\setminus\zero_\eps}\Scal_{g_0}h^2 \dd{\vol}_{g_0} 
        +\int_{\sph^n\setminus\zero_\eps}\Scal_{\overline{g}}h^{\frac{2n}{n-2}} \dd{\vol}_{g_0}.
        \end{align}
        Since $h$ is given by~\eqref{eq:hdefinition} and~$n\ge 3$, we have $\vert\frac{\partial h}{\partial \nu}h\vert\le C\vert\psi\vert^{\frac{n-3}{n-1}} \in L^{\infty}(\sph^n)$, so  by \eqref{eq:Zepslimit} there holds
        \be\label{eq:boundaryterm}
        \int_{\partial\zero_\eps}c_n\frac{\partial h}{\partial \nu}h \dd{\vol}_{g_0}\leq c_n\left\|\frac{\partial h}{\partial \nu}h \right\|_{\infty}\cH^{n-1}(\partial\zero_\eps)\,, 
        \ee
        which converges to zero along a suitable sequence $\eps_k \to 0$.
        Meanwhile, noting that~$h$ is uniformly bounded on~$\sph^n$ and $ h^{\frac{2n}{n-1}}\dv_{g_0} = \dv_{\overline{g}}$, the other two terms on the right-hand side in~\eqref{eq:estimate of gradient h} are uniformly bounded. 
        Therefore, letting~$\eps\to 0$ along the same sequence in~\eqref{eq:estimate of gradient h}, we see that
        \begin{align}\label{eq:L2gradient}
        c_n\int_{\sph^n\setminus\zero}\vert\nabla^{g_0}h\vert^2\dd{\vol}_{g_0}
        =&-\int_{\sph^n\setminus\zero} \Scal_{g_0} h^2\dv_{g_0}
        +\int_{\sph^n\setminus\zero} \Scal_{\overline{g}}h^{\frac{2n}{n-2}} \dd{\vol}_{g_0}<\infty\,,
        \end{align}
        and hence $h\in H^1(\sph^n\setminus\zero)$. 
        Observe that~$h\in C^\alpha(\sph^n\setminus\zero)$ and~$h=0$ pointwise on $\zero$, hence by \cite[Theorem 2.2]{Swanson-Ziemer99} $h\in H^1_0(\sph^n\setminus\zero)$. 

        Therefore, $h\in H^{1}_0(\sph^n\setminus\zero)\hookrightarrow  W^{1,p}_0(\sph^n\setminus\zero)$, for all $1\leq 2<p$. Since $\dim\zero\leq n-2$,  $\mathcal{H}^{n-p}(\zero)=0 $ for all $1\leq p<2$.
        Hence, by Proposition~\ref{prop:capacity and Hausdorff measure} 
        \be\label{eq:capacityestimate}
        \capacity_p(\zero) =0\,,\qquad\forall \; 1\leq p<2\,.
        \ee
        We thus conclude that $h\in W^{1,p}_0(\sph^n\setminus\zero)=W^{1,p}_0(\sph^n)=W^{1,p}(\sph^n)$, for $1\leq p<2$, by Propostion~\ref{prop:capacity and negligibility}. 
        In particular,~$h$ is weakly differentiable on the whole $\sph^n$ and its weak derivatives are $L^p$ functions on~$\sph^n$.  
        Now, since~$\zero$ has~$\cH^n$-measure zero,~\eqref{eq:L2gradient} implies that $h\in H^1(\sph^n)$. 
    \end{proof} 

By the characterization~\eqref{eq:yamabe quotient}, we now see that
\begin{equation}
 \int_{\sph^n}c_n|\nabla^{g_0}h|^2+\Scal_{g_0}h^2\dv_{g_0}
 \ge \YM(\sph^n,g_0)\parenthesis{\int_{\sph^n}h^{\frac{2n}{n-1}}\dv_{g_0} }^{\frac{n-2}{n}} 
 =n(n-1)\omega_n.
\end{equation}
Together with~\eqref{eq:estimate of gradient h} and~\eqref{eq:boundaryterm}, we get
\begin{equation}
 \lim_{\eps\to 0}\int_{\sph^n\setminus\zero_\eps}
 \Scal_{\overline{g}}\dv_{\overline{g}}
 \ge n(n-1)\omega_n. 
\end{equation}
We conclude from~\eqref{eq:int-Penrose} that 
\begin{equation}
 \int_{\sph^n\setminus\zero}|\Penr^{\overline{g}}\phi|^2\dv_{\overline{g}}=0
\end{equation}
and thus,~$\Penr^{\overline{g}}\phi=0$ on~$\sph^n\setminus\zero$, namely~$\phi$ is a twistor spinor on~$(\sph^n\setminus\zero, \overline{g})$.   
This in turn implies further information on the scalar curvature.
Indeed, a direct computation shows that 
\begin{equation}
 (\D^{\overline{g}})^2\phi= \frac{n\Scal_{\overline{g}}}{4(n-1)}\phi 
 \qquad \mbox{ in }\; (\sph^n\setminus\zero, \overline{g}),
\end{equation}
see e.g.~\cite[Prop A.2.1]{diracspectrum}. 
It follows that~$\Scal_{\overline{g}}=n(n-1)=\Scal_{g_0}$ on~$\sph^n\setminus\zero$.   
    
Using the characterization \eqref{eq:yamabe quotient}, combined with the definition \eqref{eq:hdefinition} and with \eqref{eq:gsdef} a direct computation shows that $h$ actually minimizes the Yamabe quotient. Then $h$ is a weak solution of \eqref{eq:Yamabe} in $H^1(\sph^n)$, with $\Scal_{\overline{g}}\equiv n(n-1).$

Note that~$h \in C^\alpha(\sph^n)$, hence elliptic regularity theory gives $h\in C^\infty(\sph^n)$.
Moreover, the strong maximum principle implies  that~$h>0$ on~$\sph^n$ and~$\zero(\psi)=\emptyset$. 

Now the metric~$\overline{g}$ is a smooth \emph{Riemannian} metric on~$\sph^n$ with constant scalar curvature~$\Scal_{\overline{g}}=n(n-1)=\Scal_{g_0}$. 
A theorem of Obata~\cite{obata1971theconjectures} implies that there exists an isometry
\begin{equation}
 f\colon (\sph^n,\overline{g})\to (\sph^n, g_0)
\end{equation}
that is,~$f^*g_0=\overline{g}=h^{\frac{4}{n-2}}g_0$.
Then
\begin{align}
 \dv_{f^*g_0}=\det(\dd f)\dv_{g_0}= h^{\frac{2n}{n-2}}\dv_{g_0}
 \Longrightarrow
 h=\parenthesis{\det(\dd f)}^{\frac{n-2}{2n}}. 
\end{align}
Now the spinor~$\phi\in\Sigma_{\overline{g}}\sph^n$ is an eigenspinor of eigenvalue~$\frac{n}{2}$ as well as a twistor spinor, hence a~$(-1/2)$-Killing spinor. 
These properties are preserved by isometries. 
In particular, the spinor~$F\circ \phi\circ f^{-1}$ coincides with a~$-\frac{1}{2}$-Killing spinor~$\Psi\in\Killing(g_0;-\frac{1}{2})$ which has constant length:~$|\Psi|\equiv\parenthesis{\frac{n}{2}}^{\frac{n-1}{2}}$ (see~\eqref{eq:confspinor}). 
Then~$\phi=F^{-1}\circ\Psi\circ f\equiv f^*\Psi\in \Gamma(\Sigma_{\overline{g}}\sph^n)$ and 
\begin{equation}
 \psi=h^{\frac{n-1}{n-2}}\beta^{-1}(\phi) 
 =\parenthesis{\det(\dd f)}^{\frac{n-1}{2n}}\beta^{-1}(f^*\Psi) 
  \in \Gamma(\Sigma_{g_0}\sph^n).
\end{equation}
This concludes the proof. 
\end{proof}

\begin{remark}
 Similarly to the previous section, we can explicitly compute the length function~$|\psi|$, thanks to the classification theory for the Yamabe equation. 
 Indeed, let~$h$ be a positive solution of~\eqref{eq:Yamabe}, i.e. 
 \begin{equation}
  -c_n\Delta_{g_0} h+\Scal_{g_0}h
  =\Scal_{\overline{g}} h^{\frac{n+2}{n-2}}
  \qquad \mbox{ on } \quad (\sph^n,g_0), 
 \end{equation}
 with~$c_n=4\frac{n-1}{n-2}$,~$\Scal_{g_0}=n(n-1)$, and~$\Scal_{\overline{g}}=n(n-1)$.
 Using the stereographic projection~\eqref{eq:stereographic projection} and~\eqref{eq:inverse of stereographic projection}, the induced metric~$\pi^* g_0$ has constant scalar curvature~$\Scal_{\pi^*g_0}=n(n-1)$.  
 Then the function~$\pi^*h=h\circ \pi\colon \R^n\to\R$ solves the equation 
 \begin{equation}
  -c_n\Delta_{\pi^*g_0} (\pi^*h)+\Scal_{\pi^*g_0}(\pi^*h)
  =\Scal_{\overline{g}} h^{\frac{n+2}{n-2}}
  \qquad \mbox{ on } \quad (\R^n,\pi^*g_0).
 \end{equation}
 Moreover, since the flat Euclidean metric~$g_{\R^n}$ is conformal to~$\pi^*g_0$, the function
 \begin{equation}
  u\coloneqq \parenthesis{\frac{2}{1+|x|^2}}^{\frac{n-2}{2}} (h\circ \pi) \colon \R^n\to\R 
 \end{equation}
 is a solution to the equation
 \begin{equation}\label{eq:Yamabe-Euclidean}
  -c_n\Delta_{\R^n} u=\Scal_{\overline{g}} u^{\frac{n+2}{n-2}}, 
  \qquad \mbox{ on }\quad (\R^n,g_{\R^n}).
 \end{equation}
 For~$\Scal_{\overline{g}}=n(n-1)$, the solutions of~\eqref{eq:Yamabe-Euclidean} are explicitly known from~\cite[page 211]{GidasNiNirenberg1979symmetry}, \cite[Chapter III-4]{struwevariational}: there exist~$\lambda_0$ and~$x_0\in\R^n$ such that
 \begin{equation}   
  u(x)=\parenthesis{\frac{2\lambda}{\lambda^2+|x-x_0|^2}}^{\frac{n-2}{2}}.
 \end{equation}
 This determines the length  of the solution~$\psi$: for any~$y\in\sph^n$, which is projected to~$p(y)\in\R^n$ via~\eqref{eq:stereographic projection}, 
 \begin{equation}\label{eq:length of general psi}
  |\psi(y)|
  =\parenthesis{\frac{n}{2}}^{\frac{n-1}{2}} h(y)^{\frac{n-1}{n-2}}
  =\parenthesis{\frac{n}{2}\frac{\lambda(1+|p(y)|^2)}{\lambda^2+|p(y)-x_0|^2}}^{\frac{n-1}{2}}.    
 \end{equation}
\end{remark}

\medskip 
\begin{proof}[Proof of Corollary~\ref{cor:bubbles on Rn}]
We can now give a quite explicit formula for the solutions of~\eqref{eq:nld} on~$\R^n$. 
Via the stereographic projection~$\pi$ in~\eqref{eq:inverse of stereographic projection}, the pull-back of the~$-\frac{1}{2}$-Killing spinors have the form
\begin{equation}\label{eq:Killing spinors on Rn}
 \widetilde{\Psi}(x)
 =\parenthesis{\frac{2}{1+|x|^2}}^{\frac{n}{2}}\parenthesis{\mathds{1}-\gamma_{_{\R^n}}(\vec{x})}\widetilde{\Phi}_0
\end{equation}
where~$\mathds{1}$ denotes the identity endomorphism of the spinor bundle~$\Sigma_{g_{\R^n}}\R^n$,~$\gamma_{_{\R^n}}(\vec{x})$ denotes the Clifford multiplication by the position vector~$\vec{x}$, and~$\widetilde{\Phi}_0\in \C^N$ is a constant complex~$N$-vector.  
Formula \eqref{eq:Killing spinors on Rn} is used in the study of the spinorial Yamabe problem and of critical Dirac equations on manifolds, see e.g. \cite{spinorialanalogue,ammannmass,Isobecritical}, to construct suitable test spinors.

Recall that the~$-1/2$-Killing spinors on~$(\sph^n, g_0)$ constitute a linear space of dimension~$N=2^{[n/2]}$ (see Proposition \ref{prop:killing-on-spheres}), thus the spinors of the form~\eqref{eq:Killing spinors on Rn} are their conformal image on the Euclidean space $(\R^n,g_{\R^n})$, via stereographic projection.

The other solutions of~\eqref{eq:nld} on~$(\R^n, g_{\R^n})$ are given by the transformations under conformal diffeomorphisms of~$(\R^n, g_{\R^n})$.
First consider the composition of translations and scalings: for~$x_0\in\R^n$ and~$\lambda\in\R_+$, define~$f_{x_0,\lambda}\colon \R^n \to \R^n$ by
\begin{equation}
 f_{x_0,\lambda}(x)\coloneqq \frac{x- x_0}{\lambda}. 
\end{equation}
Then~$f_{x_0,\lambda}^*g_{\R^n}=\lambda^{-1}g_{\R^n}$. 
The corresponding transformation of~\eqref{eq:Killing spinors on Rn} is given by
\begin{align}
 \psi(x)
 =&\beta_{\lambda^{-2}g_{\R^n}, g_{\R^n}} F_{x_0,\lambda}^{-1}\widetilde{\Psi}(f_{x_0,\lambda}(x)) \\
 =& \parenthesis{\frac{2\lambda}{\lambda^2+|x-x_0|^2}}^{\frac{n}{2}}
 \beta_{\lambda^{-2}g_{\R^n}, g_{\R^n}} F_{x_0,\lambda}^{-1}
 \parenthesis{\mathds{1}-\gamma_{_{\R^n}}\parenthesis{\frac{x-x_0}{\lambda}}}
 \widetilde{\Phi}_0.
\end{align}
Note that~$\beta_{\lambda^{-2}g_{\R^n}, g_{\R^n}} F_{x_0,\lambda}^{-1}$ can actually be taken as the identity, for the following reasons. 
Note that~$P_{\SO}(\R^n, g_{\R^n})=\R^n\times \SO(n)$ is the product bundle.
Using the notation from Section~\ref{sec:covariance} and~\ref{sect:conformal diffeomorphism}, we see that~$b_{g_{\R^n}, \lambda^{-2}g_{\R^n}}\circ \SO(f)\colon \R^n\times \SO(n)\to \R^n\times \SO(n)$ is given by
\begin{equation}
 (x, (v_1,\cdots, v_n))\mapsto (f_{x_0,\lambda}(x), (v_1,\cdots, v_n))
\end{equation}
which is the identity on~$\SO(n)$.
Thus its lift to the~$\Spin(n)$-principal bundles is also the identity map on the ~$\Spin(n)$ components.
As a consequence, the spinors of~$\Sigma_{g_{\R^n}}\R^n=\R^n\times \C^N$, which can be viewed as~$\C^N$-valued functions, are invariant under~$\beta_{\lambda^{-2}g_{\R^n}, g_{\R^n}} F_{x_0,\lambda}^{-1}$.
Therefore,
\begin{equation}
 \psi(x)
 =\parenthesis{\frac{2\lambda}{\lambda^2+|x-x_0|^2}}^{\frac{n}{2}}
 \parenthesis{\mathds{1}-\gamma_{_{\R^n}}\parenthesis{\frac{x-x_0}{\lambda}}}
 \widetilde{\Phi}_0.
\end{equation}
The length function of~$p^*\psi$ is exactly given by~\eqref{eq:length of general psi}, provided the constant vector~$\widetilde{\Phi}_0$ is chosen with the right norm:~$|\widetilde{\Phi}_0|=\frac{1}{\sqrt{2}}\parenthesis{\frac{n}{2}}^{\frac{n-1}{2}}$. 

Second, note that the rotations do not generate new solutions: their conformal transformations result in new choices of the parameters~$\lambda>0$,~$x_0\in \R^n$ and~$\widetilde{\Phi}_0\in \C^N$.
For example, consider a rotation~$A\in \SO(n)$, which is an isometry of~$(\R^n, g_{\R^n})$. 
Denote the induced map on spinor bundles by~$F_A\colon \Sigma \R^n,  \to\Sigma \R^n$. 
Note that~$F_A\parenthesis{\gamma_{_{\R^n}}(v)\psi}= \gamma_{_{\R^n}}(Av)F_A(\psi)$.
The pull-back of~$\psi$ under~$A$ is
\begin{align}
 A^*\psi(z)
 =& F_A^{-1}(\psi(Az))
 =\parenthesis{\frac{2\lambda}{\lambda^2+|Az-x_0|^2}}^{\frac{n}{2}}
 F_A^{-1}
 \parenthesis{\mathds{1}-\gamma_{_{\R^n}}\parenthesis{\frac{Az-x_0}{\lambda}}}
 \widetilde{\Phi}_0\\
 =&\parenthesis{\frac{2\lambda}{\lambda^2+|z-A^{-1}x_0|^2}}^{\frac{n}{2}}
 F_A^{-1}
 \parenthesis{\mathds{1}-\gamma_{_{\R^n}}\parenthesis{\frac{z-A^{-1}x_0}{\lambda}}}
 F_A^{-1}\widetilde{\Phi}_0
\end{align}
which is the solution parametrized by~$\lambda>0$,~$z_0=A^{-1}x_0\in\R^n$ and~$F_A^{-1}\widetilde{\Phi}_0\in\C^N$.

\end{proof} 
We now see that the ground state solutions of~\eqref{eq:nld} on~$(\R^n, g_{_{\R^n}})$ can be parameterized by~$\widetilde{\Phi}_0\in \C^N$ with~$|\widetilde{\Phi}|=\frac{1}{\sqrt{2}}\parenthesis{\frac{n}{2}}^{\frac{n-1}{2}}$, and~$x_0\in \R^n$,~$\lambda\in\R_+$. 
Hence they form a space of real dimension
\begin{equation}
 (2N-1)+n+1=2^{[\frac{n}{2}]+1}+n.
\end{equation}

We remark that, here we do not consider the reflections and inversions of~$\R^n$, which are also conformal, since they are orientation reversing and hence do not lift to the~$\Spin(n)$ level. 
However, since~$\Sigma\R^n=\R^n\times\C^N$ is trivial and the spinors are simply~$\C^N$ valued functions\footnote{This is to identify the spinor bundles associated to different spin structures.}, one can consider the corresponding transformations induced on the spinors. 
By a similar argument as above, one can find that they do not give rise to new solutions. 

 \section{On the Hausdorff dimension of the nodal set}\label{sec:nodalset}
 
 This section is devoted to the proof of Theorem \ref{thm:nodalset}.  
 We prove that, around a zero, a solution of \eqref{eq:nld} can be expanded as a harmonic spinor with homogeneous polynomial components, plus higher order terms. 
 Such a decomposition is the spinorial counterpart of some results by Caffarelli and Friedman in \cite{Caf-Fried-JDE79,Caf-Fried-JDE85}.
 
 We treat first the case where the leading order polynomial is of degree one and then we turn to case of higher degrees. 
 In the latter case we exploit the fact that if a solution to \eqref{eq:nld} vanishes at order $\beta>1$ at some point $x_0$, then $x_0$ must be in the critical set, i.e., $\nabla\psi(x_0)=0$.
 
 By conformal equivalence and by invariance of the Hausdorff dimension under diffeomorphisms, we equivalently study the equation on $\R^n$, that is, with respect to the \emph{Euclidean} metric,
\be\label{eq:euclideannld}
\D\psi=\vert\psi\vert^{2/(n-1)}\psi\,,\qquad\mbox{on $\R^n$.}
\ee
Moreover, it is not restrictive to look at a solution defined on the unit ball $B_1=B_1(0)\subseteq\R^n$.

Let $\psi\in C^{1,\alpha}(B_1,\C^N)$ be a solution to \eqref{eq:euclideannld}. 
Our goal is to prove that the nodal set 
\be\label{eq:ballnodalset}
\zero:=\{ x\in B_1\,:\,\psi(x)=0\}
\ee
has Hausdorff dimension at most~$n-2$. 
\begin{remark}
Since we want to deal with measure-theoretic properties of the nodal set of solutions, it is more convenient to work with real-valued spinors rather than complex-valued ones.
Thus we identify~$\C^N$ with~$\R^{2N}$ and assume~$\psi\in C^{1,\alpha}(B_1, \R^{2N})$ is a solution of~\eqref{eq:euclideannld}, which is now a system consisting of~$2N$ differential equations with real coefficients. 
\end{remark}
\subsection{Expansion of the spinor near a zero}
 In this section we prove a decomposition result for solutions to \eqref{eq:euclideannld} in $B_1$, analogous to the case of second order elliptic equations treated in \cite{Caf-Fried-JDE79,Caf-Fried-JDE85}.

The Dirac operator $\D$ can be expressed as
\begin{equation}
 \D=\alpha\cdot\nabla=\sum^n_{j=1}\alpha_j\p_j\,,
\end{equation}
where the $\alpha_j$ are $2N\times 2N$ matrices satisfying the anti-commutation Clifford relations  
\begin{equation}
 \alpha_j\alpha_k+\alpha_k\alpha_j=-2\delta_{jk}\Id_{2N}\, ,
\end{equation} 
and~$\alpha = (\alpha_1, \cdots, \alpha_n)$.  
Since $\D^2=(-\Delta)\Id_{2N}$, the Green function of $\D$ in $\R^n$ can be expressed as 
\be\label{eq:greendirac}
G(x,y)=\D_x((2-n)^{-1}\omega^{-1}_{n-1}\vert x-y\vert^{-(n-2)}\Id_{2N})=\frac{\alpha\cdot(x-y)}{\omega_{n-1}\vert x-y\vert^n}\Id_{2N}, 
\ee
and it verifies 
$$
\D_xG(x-y)=\delta(x-y)\Id_{2N} 
$$
in the distributional sense. Then, integrating by parts one finds the representation formula
\be\label{eq:repformula}
\psi(x)=\int_{B_1}G(x-y)\D\psi(y)\dd{y}\, +\int_{\p B_1}(\alpha\cdot y)G(x-y)\psi(y) \dd{S(y)}\,=:I_1+I_2\,,
\ee
where we abbreviated~$\alpha\cdot y \equiv \sum_j y^j\alpha_j$. 

\begin{lemma}\label{lem:basicdecomp}
Suppose $\psi$ satisfies
\be\label{eq:diracbeta}
\vert\D\psi\vert\leq C_\beta \vert x\vert^\beta\,,\qquad\mbox{on $B_1$,}
\ee
and that $\beta>0$ is \emph{not} an integer.
Then there exists $0<R\leq 1$ such that
\be\label{eq:decomp}
\psi(x)=P(x)+\Gamma(x),\qquad\mbox{on $B_R$,}
\ee
for some~$P,\Gamma:B_R\to\R^{2N}$, where the components of $P$ are harmonic polynomials of degree $[\beta]+1$, and 
\begin{equation}\label{eq:Gamma}
 \vert\Gamma(x)\vert\leq C'_\beta\vert x\vert^{\beta+1},\quad \vert \nabla\Gamma(x)\vert\leq C''_\beta\vert x\vert^{\beta}\,,\qquad\mbox{on $B_R$.}
\end{equation}
Moreover $P$ is a harmonic spinor, i.e. $\D P=0$.
\end{lemma}
\begin{proof}
We need to analyze the terms $I_1,I_2$ in \eqref{eq:repformula}. 

Recall that the Green's kernel of the Laplacian admits a power series expansion in terms of the so-called \emph{Gegenbauer polynomials} \cite[p. 148-150]{Stein-Weiss-71}
\be\label{eq:greenlaplacian}
\vert x-y\vert^{-(n-2)}=\sum_{k\geq0}\frac{1}{\vert y\vert^{n-2+k}}\vert x\vert^k C^{\tau}_k(\langle x, y\rangle)\,,\quad \tau=(n-2)/2\,,
\ee
where $C^{\tau}_k(t)$ are the Gegenbauer polynomials of indices $(k,\gamma)$, and $\vert x\vert^k C^{\tau}_k(\langle x, y\rangle)$ are homogeneous harmonic polynomials in $x$ of degree $k$. Observe that in the above formula $\langle x, y\rangle$ denotes the Euclidean scalar product of $x,y\in\R^n$.

Then by \eqref{eq:greendirac} we conclude that the Green's function of $\D$ can be rewritten as
\be\label{eq:greendiracseries}
G(x-y)=\frac{1}{\omega_{n-1}}\sum_{k\geq0}\frac{1}{\vert y\vert^{n-2+k}}\Xi_k(x,y)\,,
\ee
where the $2N\times 2N$ matrix 
\be\label{eq:Xi}
\Xi_k(x,y):=\D_x(\vert x\vert^k C^{\tau}_k(\langle x, y\rangle))
\ee 
is $\D_x$-harmonic and its components are homogeneous harmonic polynomials of degree $k-1$, recalling that $\D^2=(-\Delta) \Id_{2N}$.  

\begin{remark}\label{rmk:seriesconvergence}
Notice that the power series in \eqref{eq:greenlaplacian} is absolutely convergent in a smaller ball  $B_R\Subset B_1$. This follows from the properties of the Gegenbauer polynomials, for which we refer the reader to \cite[p. 148-150]{Stein-Weiss-71} and \cite{Kim-Kim-Rim-2012}. Indeed, there holds 
\be\label{eq:derivative-gegenbauer}
\left\vert\frac{d^j}{dt^j} C^{\tau}_k(t) \right\vert\leq C k^{2j+n-3}\,,\qquad j=0,1,2\,.
\ee
 One easily sees that 
 \[
\vert \Xi_k(x,y)\vert\leq L \left [k\vert x\vert^{k-1}C^{\tau}_k(\langle x, y\rangle)
+\vert x\vert^k\nabla_x C^{\tau}_k(\langle x, y\rangle)\right]\,,
 \]
 for some $L>0$. Then, by \eqref{eq:Xi} and \eqref{eq:derivative-gegenbauer}, one concludes that the series in \eqref{eq:greendiracseries} converges uniformly for $x\in B_R$.
\end{remark}

We estimate $I_1$, decomposing the domain of integration as follows
$$
\int_{B_1}=\int_{B_1\cap B_{(1+1/\beta)\vert x\vert}(0)}+\int_{B_1\setminus B_{(1+1/\beta)\vert x\vert}(0)}\,,
$$
and then splitting 
$$I_1=J_1+J_2$$
accordingly. We can estimate $J_1$ as follows, by \eqref{eq:greendirac}, \eqref{eq:diracbeta} and passing to polar coordinates
\be
\begin{split}
\vert J_1\vert &\leq C C_\beta\vert x\vert^\beta\int_{B_1\cap B_{(1+1/\beta)\vert x\vert}}\frac{\dd{y}}{\vert x-y\vert^{n-1}}\leq C'_\beta\vert x\vert^\beta\int_{B_{2(1+1/\beta)\vert x\vert}(x)}\frac{\dd{y}}{\vert x-y\vert^{n-1}}\\
 &=C'_\beta\vert x\vert^\beta\int_{B_{2(1+1/\beta)\vert x\vert}}\frac{\dd{z}}{\vert z\vert^{n-1}} \leq \widetilde{C}_\beta\vert x\vert^{\beta+1}, 
\end{split}
\ee
using the inclusion $B_{(1+1/\beta)\vert x\vert}(x)\subseteq B_{2(1+1/\beta)\vert x\vert}$.

We now turn to $J_2$, exploiting the expansion \eqref{eq:greendiracseries}. 
Observe that the properties of $\Xi_k(x,y)$ imply that the series converges uniformly, so that one can differentiate or integrate term by term. 
There holds
\be\label{eq:J2series}
\begin{split}
J_2
=&\sum_{k\geq 0}\int_{B_1\setminus B_{(1+1/\beta)\vert x\vert}}\frac{1}{\vert y\vert^{n-2+k}}\Xi_k(x,y)\D\psi(y)\dd{y} \\
=& \sum^{[\beta]+2}_{k=0}\int_{B_1\setminus B_{(1+1/\beta)\vert x\vert}}\frac{1}{\vert y\vert^{n-2+k}}\Xi_k(x,y)\D\psi(y)\dd{y}\\
&\quad +\sum_{k>[\beta]+2}\int_{B_1\setminus B_{(1+1/\beta)\vert x\vert}}\frac{1}{\vert y\vert^{n-2+k}}\Xi_k(x,y)\D\psi(y)\dd{y}\\
=:&\; \cA+\cB\,.
\end{split}  
\ee
Let us focus on $\cA$. Adding the sum
\be\label{eq:tildeA}
\widetilde{\cA}
=\sum^{[\beta]+2}_{k=0}\int_{B_1\cap B_{(1+1/\beta)\vert x\vert}}\frac{1}{\vert y\vert^{n-2+k}}\Xi_k(x,y)\D\psi(y)\dd{y}
=:\sum^{[\beta]+2}_{k=0}\widetilde{\cA}_k
\ee
to~$\cA$, we obtain a spinor 
\be\label{eq:p0}
P_0:=\cA+\widetilde{\cA}, 
\ee
 which is harmonic and whose components are harmonic polynomials of degree $[\beta]+1$. 
 
 Passing to polar coordinates, we can estimate the terms appearing in $\widetilde{A}$ as follows
 \be\label{eq:tildeAestimate}
 \begin{split}
  \vert\widetilde{\cA}_k\vert&\leq C_1\vert x\vert^{k-1}\int_{B_1\cap B_{(1+1/\beta)\vert x\vert}}\frac{\dd{y}}{\vert y\vert^{n-2+k-\beta}} \\
   & \leq C_2\vert x\vert^{k-1}\int^{(1+1/\beta)\vert x\vert}_0 r^{\beta+1-k}\dd{r}\leq C_2\frac{(1+1/\beta)^{\beta+2}}{\beta+2}\vert x\vert^{\beta+1}\,,
  \end{split}
  \ee
  where we used the fact that $\Xi_k(x,y)$ is $(k-1)$-homogeneous and \eqref{eq:diracbeta}.
  
 We now need to estimate the term
\be\label{eq:B}
\cB=\sum_{k>[\beta]+2}\cB_k\,,
\ee
where
\be\label{eq:B_k}
\cB_k=\int_{B_1\setminus B_{(1+1/\beta)\vert x\vert}}\frac{1}{\vert y\vert^{n-2+k}}\Xi_k(x,y)\D\psi(y)\dd{y}\,.
\ee
Using \eqref{eq:diracbeta} and the definition of $\Xi_k(x,y)$ we get
\be
\vert \cB_k\vert\leq C C_\beta \vert x\vert^{k-1}\int_{\R^n\setminus B_{(1+1/\beta)\vert x\vert}}\frac{\dd{y}}{\vert y\vert^{n-2+k-\beta}}\,.
\ee
Notice that the constant $C_\beta$ is independent of $k$. Then, passing in polar coordinates in the last integral we obtain
\be
\begin{split}
\int_{\R^n\setminus B_{(1+1/\beta)\vert x\vert}}\frac{\dd{y}}{\vert y\vert^{n-2+k-\beta}}
&=\omega_n\int^\infty_{(1+1/\beta)\vert x\vert}\frac{\dd{r}}{r^{k-\beta-1}}\\
&\leq \omega_n\frac{(1+1/\beta)^{\beta+2}}{[\beta]+1-\beta}\vert x\vert^{\beta+2}\times(1+1/\beta)^{-k}\vert x\vert^{-k}\,.
\end{split}
\ee

Combining the above observations, summing up and using \eqref{eq:B} we thus find
\be\label{eq:Bestimate}
\vert \cB\vert\leq C\frac{C_\beta}{\beta-[\beta]}\vert x\vert^{\beta+1}\,.
\ee
Observe that that $\beta-[\beta]\neq0$, as we assumed that $\beta$ is not an integer.

We are left with the term $I_2$ in \eqref{eq:repformula}. Using the expansion \eqref{eq:greendiracseries} and the fact that $\vert y\vert=1$, we see that
\be
\begin{split}\label{eq:I_2}
I_2&=\int_{\p B_1}(\alpha\cdot y)G(x-y)\psi(y) \dd{S(y)}
=\sum^{\infty}_{k=0}Q_k(x)\,,
\end{split}
\ee
 where $Q_k:B_1\to\C^N$ is ~$\D$-harmonic, i.e. $\D Q_k=0$,  and its components are homogeneous harmonic polynomials of degree $k-1$. 
 
 By \eqref{eq:greendirac} the components of the spinor
 \be\label{eq:p1}
 P_1(x):=\sum^{[\beta]+2}_{k=0}Q_k(x)
 \ee
 are harmonic polynomials of degree $[\beta]+1$, and there holds $\D P_1=0$. 
 
 The remainder term can be estimated, following \cite[p. 342-343]{Caf-Fried-JDE79}, as follows.
 Observe that 
 \[
 \vert Q_k(x)\vert\leq C\delta^{k-1}\vert x\vert^{k-1}\,,\qquad\forall\delta>1\,,
 \]
 where $C>0$ depends on $\delta$ and $\Vert\psi\Vert_{L^\infty(\partial B_1)}$. Now, if $\vert x\vert<\rho$, we have
 \[
  \left\vert \sum_{k>[\beta]+2}Q_k(x)\right\vert\leq C \sum_{k>[\beta]+2}\delta^{k-1}\vert x\vert^{k-1}\leq C'\delta^{\beta}\vert x\vert^{\beta+1}\,,
 \]
 where $C'$ depends on $C,\delta, \rho$. If $\rho<\vert x\vert<1$, then 
 \begin{align*}
  \left\vert \sum_{k>[\beta]+2}Q_k(x)\right\vert &= \left\vert I_2-\sum_{k\leq[\beta]+2}Q_k(x)\right\vert\leq C+\sum_{k\leq[\beta]+2}\vert Q_k(x)\vert \\
  & \leq C+\sum_{k\leq[\beta]+2}\delta^{k-1}\leq C''\delta^{\beta+1}\leq C''\left(\frac{\delta}{\rho} \right)^{\beta+1}\vert x\vert^{\beta+1}\, 
 \end{align*}
 for some other constant $C''>0$. Taking $\delta=5/4$ and $\rho=3/4$, we get
 \be\label{eq:remainder}
 \left\vert \sum_{k>[\beta]+2}Q_k\right\vert\leq C_\beta\vert x\vert^{\beta+1}\,.
 \ee
 Then formula \eqref{eq:decomp} follows combining \eqref{eq:p0},\eqref{eq:Bestimate},\eqref{eq:p1} and \eqref{eq:remainder}, taking 
 \[
 P:=P_0+P_1\,,
 \]
 and, by \eqref{eq:B},\eqref{eq:remainder},
 \[
 \Gamma:= \sum_{k>[\beta]+2}\left(\mathcal{B}_k+Q_k\right)\,.
 \]
 Let us focus now on gradient estimates in \eqref{eq:Gamma}. There holds 
 \be\label{eq:nablaGamma}
 \nabla\Gamma=\sum_{k>[\beta]+2}\left(\nabla\mathcal{B}_k+\nabla Q_k\right)\,.
 \ee
 Observe that the components of $\left(\nabla\mathcal{B}_k+\nabla Q_k\right)$ are homogeneous polynomials of degree $k-2$. The gradient estimate in \eqref{eq:Gamma} follows along the same line as for the proof of \eqref{eq:Bestimate} and \eqref{eq:remainder}, as the argument in Remark \eqref{rmk:seriesconvergence} shows that $\left\vert\nabla\mathcal{B}_k(x)+\nabla Q_k(x)\right\vert \leq \,C k^{n+1}\vert x\vert^{k-2}$, so that the series \eqref{eq:nablaGamma} is uniformly convergent, possibly restricting to a smaller ball $B_{R'}\Subset B_{R}\Subset B_1$.
\end{proof}

Since $\psi$ is a solution to \eqref{eq:euclideannld}, then
$$
\vert \D\psi\vert=\vert\psi\vert^{(n+1)/(n-1)}\,\qquad\mbox{on $B_1$.}
$$
Let $x_0\in B_1$ be such that $\psi(x_0)=0$. 
Without loss of generality, we assume $x_0=0$. 

\begin{lemma}\label{lem:homdecomp}
Suppose that a spinor~$\psi\in C^{1,\alpha}(B_1,\R^{2N})$ satisfies
$$
\vert\D\psi\vert\leq C\vert\psi\vert^\sigma\,,\qquad\mbox{on $B_1$,}
$$
with $C\geq0, \sigma\geq1$. 
Assume that 
\begin{equation}\label{eq:nontriviality of psi}
\psi(0)=0\,,\qquad \psi\not\equiv0\,,\quad\mbox{on $B_1$.}
\end{equation}
Then there exist $P_k,\Gamma_k:B_1\to\C^N$ such that 
\be\label{eq:homdecomp}
\psi(x)=P_k(x)+\Gamma_k(x)\,, \qquad x\in B_R\,,
\ee
for some $0<R\leq 1$. Here the components of $P_k$ are homogeneous harmonic polynomials of degree $k\geq1$, and there holds $\D P_k=0$. Furthermore, for any $0<\delta<1$ there exists a constant $C=C(\delta)>0$ such that
\be\label{eq:highorder}
\vert\Gamma_k(x)\vert\leq C\vert x\vert^{k+\delta}\,,
\qquad \vert\nabla\Gamma_k(x)\vert\leq C\vert x\vert^{k+\delta-1}\,.
\ee

\end{lemma}
\begin{proof}
If not, then we can repeatedly apply Lemma \ref{lem:basicdecomp} and conclude that $\psi$ vanishes to infinite-order at $x=0$, in the sense that~$\psi(x)=o(\vert x\vert^m)$, for any $m\in\mathbb{N}$.
The strong unique continuation principle \cite[Corollary to Theorem 1]{Kim-ProcAms95} implies that $\psi\equiv 0$, contradicting~\eqref{eq:nontriviality of psi}.  
\end{proof}

 \subsection{Dimension estimates for the nodal set: proof of Theorem \ref{thm:nodalset}} 

As before, consider a non-trivial $C^{1,\alpha}$ solution $\psi:B_1\to\R^{2N}$ of
\[
\D\psi=\vert\psi\vert^{2/(n-1)}\psi\,,
\]
and let  $\zero=\{x\in B_1:\psi(x)=0\}$ be its nodal set.

For each given $x_0\in\zero$, the spinor $\psi$ admits a decomposition
\be\label{eq:localdecomp}
\psi(x)=P_k(x-x_0)+\Gamma_k(x-x_0)\,,\qquad\mbox{in $B_R(x_0)$,}
\ee
as in Lemma \ref{lem:homdecomp}, where $k\geq1$. Then we have
\be\label{eq:split-nodal-set}
\zero=\zero_1\cup\zero_{\geq2}\,,
\ee
where the set $\zero_1$ consists of points in $\zero$ for which the leading order polynomial term in \eqref{eq:localdecomp} is of first order, and $\zero_{\geq2}=\zero\setminus\zero_1$. We estimate the Hausdorff dimension of those sets separately, so that the proof of Theorem \ref{thm:nodalset} follows combining Propositions \ref{prop:k=1-estimate} and \ref{prop:k>1-estimate}.

\smallskip

\begin{proposition}\label{prop:k=1-estimate}
The set $\zero_1$ in \eqref{eq:split-nodal-set} has Hausdorff dimension at most $n-2$.
\end{proposition}
\begin{proof}
Take $x_0\in\zero_1$. For simplicity we assume that $x_0=0$. Our aim is to prove that there exists $\rho>0$ such that $\zero\cap B_\rho$ is contained in a rectifiable subset of dimension at most $n-2$. 

Since $k=1$,  we have $P_1=(P^1_1, \cdots, P^{2N}_1)$, where each~$P^j_1$ is a homogenous polynomial of degree one, namely a linear function. 

As $P_1\neq 0$, the vector space~$\Span_\R\{P^1_1, \cdots , P^{2N}_1\}$ is non-trivial. 

We claim that the vector space ~$\Span_\R\{P^1_1, \cdots , P^{2N}_1\}$ cannot be one-dimensional.
Arguing by contradiction, suppose that there exists a non-zero linear~$p(x^1, \cdots, x^n)$ and constants~$c^1, \cdots, c^{2N}\in \R$ such that
\begin{equation}
 P^j_1=c^j p ,\quad 1\le j \le 2N
\end{equation}
and at least one~$c^j$ is non-zero.
Note that~$\nabla \Gamma(0)=0$. 
Then at~$x_0=0\in B_1$,
\begin{equation}
 \D\psi(0)
 =\sum_{1\le\alpha\le n} \gamma(e_\alpha)\nabla_{e_\alpha}\psi(0)
 =\sum_{1\le\alpha\le n} \gamma(e_\alpha)\nabla_{e_\alpha}P_1(0).
\end{equation}
Now since~$p(x^1,\cdots, x^n)$ is linear, we may perform a linear transformation on~$B_1(0)\subset \R^n$ such that
~$p(x^1,\cdots, x^n)= x^1$, and hence~$\nabla_{e_\alpha}p=\delta_{1\alpha}$. 
Consequently, at the origin 
\begin{equation}
 \D\psi(0)= \sum_{1\le\alpha\le n}\gamma(e_\alpha)
 \begin{pmatrix}
  c^1 \\\vdots \\ c^{2N}
 \end{pmatrix}\delta_{1\alpha}
 =\gamma(e_1)\begin{pmatrix}
  c^1 \\\vdots \\ c^{2N}
 \end{pmatrix}. 
\end{equation}
On the other hand,  equation~\eqref{eq:euclideannld} implies that~$\D\psi(0)=0$. 
Since~$\gamma(e_1)$ in invertible, we are led to~$c^1=\cdots=c^{2N}=0$, a contradiction. 

Therefore, the vector space~$\Span_\R\{P^1_1, \cdots , P^{2N}_1\}$ is at least two-dimensional.
Suppose that~$P^1_1,P^2_1$ are linearly independent, then so are their gradients $\nabla P^1_1,\nabla P^2_1$. 
Note that
\[
\zero=\{x\in B_R :\psi(x)=0 \}\subseteq\{x\in B_R : \psi^1(x)=0,\psi^2(x)=0\}=:\Omega\,,
\]
and $\nabla\psi^1(0)=\nabla{P}^1_1(0), \nabla\psi^2(0)=\nabla{P}^2_1(0)$ are linearly independent.
By the implicit function theorem, there exists $\rho>0$ such that $\Omega\cap B_\rho$ is a submanifold of dimension $n-2$, as desired.
\end{proof}
  
We now deal with the set $\zero_{\geq 2}$, that is, we consider zeroes of the spinor for which the leading order polynomial in \eqref{eq:localdecomp} is of order $k\geq2$. In this case the dimension estimate follows by the analogous result in \cite{Caf-Fried-JDE85}.

\begin{proposition}\label{prop:k>1-estimate}
The set $\zero_{\geq2}$ has Hausdorff dimension at most $n-2$.
\end{proposition}
\begin{proof}
It is immediate to see that 
\be
\zero_{\geq2}=\{x_0\in B_R : \psi(x_0)=0, \nabla\psi(x_0)=0\,,\psi(x)=P_k(x-x_0)+\Gamma_k(x-x_0)\,, k\geq2\}, 
\ee
where  $P_k$ and $\Gamma_k$ are as in Lemma \ref{lem:homdecomp}. Observe that the components of the spinors $P_k$ are harmonic polynomials, and that 
\be
\begin{split}
&\{x_0\in B_R : \psi(x_0)=0, \nabla\psi(x_0) =0\,, \psi(x)=P_k(x-x_0)+\Gamma_k(x-x_0)\,, k\geq2\} \\
&= \bigcap^N_{j=1}\{x_0\in B_R : \psi^j(x_0)=0, \nabla\psi^j(x_0)=0, \psi^j(x) =P^j_k(x-x_0)+\Gamma^j_k(x-x_0)\,, k\geq2\}\,,
\end{split}
\ee
where $\psi=(\psi^1,\cdots,\psi^N)$. 
Then we are led to estimate the dimension of the sets
\be\label{eq:N_j}
N_j:=\{x_0\in B_R : \psi^j(x_0)=0, \nabla\psi^j(x_0)=0, \psi^j(x) =P^j_k(x-x_0)+\Gamma^j_k(x-x_0)\,, k\geq2\}\,,
\ee
where $j=1,\cdots,N$, as clearly $\zero_{\geq 2}=\cap_{j=1}^N N_j$.

The desired estimate $\dim_{\mathcal{H}} N_j\leq n-2$ follows from \cite[Theorem 3.1]{Caf-Fried-JDE85}. 
Indeed, in that paper the authors show that $\dim_{\mathcal{H}}(S)\leq n-2$, $S:=\{x\in\Omega\,:\, u(x)=0,\nabla u(x)=0\}$ being the singular set of solutions to second order elliptic equations of the form
\[
\Delta u=f(x,u,\nabla u)\qquad \mbox{in $\Omega$}\,,
\]
where $\Omega\subseteq\R^n$ is an open set and 
\[
\vert f(x,u,\nabla u)\vert\leq A \vert u\vert^\alpha+B\vert \nabla u\vert^\beta\,,
\]
for some $A,B>0$ and $\alpha,\beta\geq1$.
For such functions, they proved a decomposition result \cite[Theorem 1.2]{Caf-Fried-JDE85} analogous to \eqref{eq:homdecomp}. 
Starting from such a decomposition, they obtained  cusp-like estimates \cite[Theorem 2.1]{Caf-Fried-JDE85} and then conclude the proof of \cite[Theorem 3.1]{Caf-Fried-JDE85}. The result also applies in our case for the sets $N_j$ in \eqref{eq:N_j}. Indeed, the proof of \cite[Theorem 3.1]{Caf-Fried-JDE85} ultimately relies on the decomposition  \cite[Forumula 1.8]{Caf-Fried-JDE85}, whose analogue in our case is given by \eqref{eq:homdecomp} in Lemma \ref{lem:homdecomp}. Starting from that results one can apply the mentioned argument of \cite{Caf-Fried-JDE85} to each component $\psi^j$, thus concluding the proof.

\end{proof}

\end{document}